\documentclass[11pt, letterpaper, oneside,reqno]{amsart}

\usepackage{latexsym, amsmath, amssymb, amsfonts, amscd,bm}
\usepackage{amsthm}
\usepackage{t1enc}
\usepackage[mathscr]{eucal}
\usepackage{indentfirst}
\usepackage{graphicx, pb-diagram}
\usepackage{adjustbox}
\usepackage{fancyhdr}
\usepackage{fancybox}
\usepackage{enumerate}
\usepackage{color}
\usepackage[all]{xy}
\usepackage{hyperref}
\usepackage{tikz-cd}
\usepackage{mathtools}
\usepackage{wrapfig}
\usepackage{url}
\usepackage{float}
\usepackage{bm}

\theoremstyle{plain}
\newtheorem{thm}{Theorem}[section]

\newtheorem{prop}[thm]{Proposition}

\newtheorem{lemma}[thm]{Lemma}
\newtheorem{cor}[thm]{Corollary}

\newtheoremstyle{underline}
{}        
{}              
{}              
{}    
{\large}              
{:}             
{1mm}         
{{\underline{\thmname{#1}\thmnumber{ #2}}}}  

\theoremstyle{underline}
\newtheorem*{claim*}{Claim}

\theoremstyle{definition}
\newtheorem{defi}[thm]{Definition}

\theoremstyle{remark}
\newtheorem{remark}[thm]{Remark}
\newtheorem{ex}[thm]{Example}
\newtheorem{exs}[thm]{Examples}

\newtheorem*{ack}{Acknowledgements}

\newcommand{\RR}{\ensuremath{\mathbb R}}

\definecolor{forest}{rgb}{0,0.5,0}

\begin{document}
	
	\title{The Poisson saturation of coregular submanifolds}
	
	\author{Stephane Geudens}
	\address{\scriptsize{KU Leuven, Department of Mathematics, Celestijnenlaan 200B, 3001 Leuven, Belgium \newline Current address: Max-Planck-Institut für Mathematik, Vivatsgasse 7, 53111 Bonn, Germany}}
	
	\email{stephane\_geudens@hotmail.com}
	
	\begin{abstract}
		This paper is devoted to coregular submanifolds in Poisson geometry. We show that their local Poisson saturation is an embedded Poisson submanifold, and we give a normal form for this Poisson submanifold around the coregular submanifold. This result recovers the normal form around Poisson transversals, and it yields Poisson versions of some normal form/rigidity results around constant rank submanifolds in symplectic geometry. As an application, we prove a uniqueness result concerning coisotropic embeddings of Dirac manifolds in Poisson manifolds. We also show how our results generalize to the setting of coregular submanifolds in Dirac geometry.
	\end{abstract}

	\maketitle
	
	\setcounter{tocdepth}{1} 
	\tableofcontents
	
	\section*{Introduction}
	A well-known result in symplectic geometry is Weinstein's generalized Darboux theorem, which states that for any embedded submanifold $X$ of a symplectic manifold $(M,\omega)$, the restriction of $\omega$ to $TM|_{X}$ determines the symplectic form $\omega$ on a neighborhood of $X$ up to symplectomorphism \cite{lagrangians}.
	By contrast, given a Poisson manifold $(M,\pi)$ and any embedded submanifold $X\subset M$, one should not expect $\pi$ to be determined, up to neighborhood equivalence, by its restriction $\pi|_{X}$. For instance, the origin in $\RR^{2}$ is a fixed point for both the zero Poisson structure and the Poisson structure $\pi=(x^{2}+y^{2})\partial_{x}\wedge\partial_{y}$, which are clearly not diffeomorphic around $(0,0)$. 
	
	In order for the restriction $\pi|_{X}$ to determine $\pi$ around $X$, the ambient Poisson manifold needs to satisfy a minimality condition with respect to $X$. Since $\pi|_{X}$ only contains information in the leafwise direction along $X$, we are led to consider the \emph{saturation} of $X\subset(M,\pi)$, i.e. the union of the symplectic leaves that intersect $X$. Clearly, the saturation of $X$ fails to be smooth in general; the purpose of this note is to single out a class of embedded submanifolds $X\subset(M,\pi)$ whose saturation is smooth near $X$, in a sense that will be made precise later. Since the saturation $Sat(X)$ of $X\subset(M,\pi)$ is traced out by following Hamiltonian flows starting at points of $X$ in directions normal to $X\subset M$, it is natural to impose the following regularity condition on $X$.  
	
	\vspace{0.15cm}
	\textbf{Definition.}~We call an embedded submanifold $X$ of a Poisson manifold $(M,\pi)$ \textbf{coregular} if the map $\mathrm{pr}\circ\pi^{\sharp}:T^{*}M|_{X}\rightarrow TM|_{X}/TX$ has constant rank. 
	
	\vspace{0.15cm}
	It is equivalent to require that the $\pi$-orthogonal $TX^{\perp_{\pi}}:=\pi^{\sharp}(TX^{0})$ has constant rank. Extreme examples are transversals and Poisson submanifolds of $(M,\pi)$, and we show that any coregular submanifold $X\subset(M,\pi)$ is obtained by intersecting such submanifolds. Note that if $\pi$ is symplectic, then any submanifold of $(M,\pi)$ is coregular.
	
	\vspace{0.15cm}
	The main result of Section \ref{sec:first} is the fact that the saturation of a coregular submanifold $X\subset(M,\pi)$ is smooth around $X$, in the following sense.
	
	\vspace{0.15cm}
	\textbf{Theorem A.} \emph{If $X\subset(M,\pi)$ is a coregular submanifold, then there exists a neighborhood $V$ of $X$ such that the saturation of $X$ inside $(V,\pi|_{V})$ is an embedded Poisson submanifold.}
	
	\vspace{0.15cm}
	We will refer to this Poisson submanifold as the \emph{local Poisson saturation} of $X$. The proof of Theorem A relies on some contravariant geometry and some results concerning dual pairs in Poisson geometry. 
	
	\vspace{0.15cm}
	Sections \ref{sec:second} and \ref{sec:third} are devoted to the construction of a normal form for the local Poisson saturation of a coregular submanifold. In Section \ref{sec:second}, we introduce the local model; it is defined on the total space of the vector bundle $(TX^{\perp_{\pi}})^{*}$, and it depends on two choices:
	\begin{enumerate}
		\item A choice of complement $W$ to $TX^{\perp_{\pi}}$ inside $TM|_{X}$. Such a choice yields an inclusion $j:(TX^{\perp_{\pi}})^{*}\hookrightarrow T^{*}M|_{X}$.
		\item A choice of closed two-form $\eta$ on a neighborhood of $X$ in $(TX^{\perp_{\pi}})^{*}$, with prescribed restriction $\eta|_{X}=-\sigma-\tau$ along the zero section $X\subset(TX^{\perp_{\pi}})^{*}$. Here $\sigma\in\Gamma(\wedge^{2}TX^{\perp_{\pi}})$ and $\tau\in\Gamma(T^{*}X\otimes TX^{\perp_{\pi}})$ are bilinear forms defined by
		\begin{align*}
			&\sigma(\xi_{1},\xi_{2})=\pi\big(j(\xi_{1}),j(\xi_{2})\big),\nonumber\\
			&\tau\big((v_{1},\xi_{1}),(v_{2},\xi_{2})\big)=\langle v_{1},j(\xi_{2})\rangle - \langle v_{2},j(\xi_{1})\rangle,
		\end{align*}
		for $\xi_{1},\xi_{2}\in\big(T_{x}X^{\perp_{\pi}}\big)^{*}$ and $v_{1},v_{2}\in T_{x}X$.
	\end{enumerate}
	
	To such a complement $W$ and closed extension $\eta$, we associate a Poisson structure $\big(U,\pi(W,\eta)\big)$ on a neighborhood $U$ of $X\subset(TX^{\perp_{\pi}})^{*}$. It is described conveniently using the language of Dirac geometry, see e.g. \cite[\S 2]{dirac} for definitions and notation. The construction goes as follows: pull back the Dirac structure $L_{\pi}$ defined by the Poisson structure $\pi$ under $i:X\hookrightarrow (M,L_{\pi})$, then pull back once more by the bundle projection $\mathrm{pr}:(TX^{\perp_{\pi}})^{*}\rightarrow (X,i^{!}L_{\pi})$ and gauge transform by the closed extension $\eta$. The obtained Dirac structure $\big(\mathrm{pr}^{!}(i^{!}L_{\pi})\big)^{\eta}$ is Poisson on a neighborhood $U$ of $X\subset(TX^{\perp_{\pi}})^{*}$. This Poisson structure, denoted by $\big(U,\pi(W,\eta)\big)$, is the local model for the local Poisson saturation of $X\subset(M,\pi)$, as we prove in Section \ref{sec:third}.
	
	\vspace{0.15cm}
	\textbf{Theorem B.}
	\emph{Let $X\subset(M,\pi)$ be a coregular submanifold. A neighborhood of $X$ in its local Poisson saturation is Poisson diffeomorphic with the local model $\big(U,\pi(W,\eta)\big)$.}
	\vspace{0.15cm}
	
	The proof of this result goes along the same lines as the proof of the normal form around Poisson transversals \cite{transversals}, using dual pairs in Dirac instead of Poisson geometry. 
	
	Since the local model $\big(U,\pi(W,\eta)\big)$ is constructed out of the restriction $\pi|_{X}$, Theorem B shows that the local Poisson saturation of a coregular submanifold $X$ is determined by the restriction $\pi|_{X}$, up to Poisson diffeomorphism around $X$. We thus obtain a Poisson version of Weinstein's generalized Darboux theorem in symplectic geometry. In general, one needs the full information of $\pi|_{X}$ in order to determine the local Poisson saturation of $X$. However, there are distinguished coregular submanifolds $X$ for which only part of this information is required, as we show in Section \ref{sec:fourth}.
	
	\vspace{0.15cm}
	In Section \ref{sec:fourth}, we specialize our normal form to some particular classes of coregular submanifolds. These allow for a good choice of complement $W$ and/or closed extension $\eta$, and as such our normal form becomes more explicit. Most notably, we obtain statements concerning the following types of submanifolds, i) and ii) being particular instances of iii):
	\begin{enumerate}[i)]
		\item \underline{Poisson transversals}: We recover the normal form theorem around Poisson transversals, which was established in \cite{transversals}, \cite{eulerlike}.
		\item \underline{Coregular coisotropic submanifolds:} We obtain a Poisson version of Gotay's normal form theorem from symplectic geometry \cite{gotay}, which shows that the local Poisson saturation of a coregular coisotropic submanifold $i:X\hookrightarrow(M,\pi)$ is determined, up to Poisson diffeomorphism around $X$, by the pullback Dirac structure $i^{!}L_{\pi}$.
		
		
		\item \underline{Coregular pre-Poisson submanifolds:} We obtain a Poisson analog of Marle's constant rank theorem from symplectic geometry \cite{marle}. Loosely speaking, the result shows that the local Poisson saturation of a coregular pre-Poisson submanifold $i:X\hookrightarrow(M,\pi)$ is determined, up to Poisson diffeomorphism around $X$, by the pullback Dirac structure $i^{!}L_{\pi}$ and the restriction of $\pi$ to $(TX^{\perp_{\pi}})^{*}/(TX^{\perp_{\pi}}\cap TX)^{*}$.
	\end{enumerate}
	
	\vspace{0.15cm}
	In Section \ref{sec:fifth}, we present an application of our normal form specialized to the case of coregular coisotropic submanifolds. We address the problem of embedding a Dirac manifold $(X,L)$ coisotropically into a Poisson manifold $(M,\pi)$, which was considered before in \cite{marco} and \cite{wade}. Existence of coisotropic embeddings is settled in \cite{marco}, where one shows that such an embedding exists exactly when $L\cap TX$ has constant rank. An explicit construction of the Poisson manifold $(M,\pi)$ is given in that case; another construction appears in \cite{wade}.  
	
	\noindent
	The uniqueness of such embeddings was conjectured in \cite{marco}, but only proved under additional regularity assumptions on $(X,L)$. Using our normal form result, we can show that any coisotropic embedding of $(X,L)$ factors through the model $(M,\pi)$
	constructed in \cite{marco}, which in turn proves the conjecture concerning the uniqueness of coisotropic embeddings.
	
	
	\vspace{0.15cm}
	In Section \ref{sec:sixth}, we discuss how our results can be generalized to the setting of coregular submanifolds in Dirac geometry. The Appendix contains a result in differential topology for which we could not find a proof in the literature.
	
	\vspace{0.15cm}
	\noindent
		\emph{Terminology and notation.} We freely use notions from Dirac geometry throughout the text, adopting terminology and notation from \cite{dirac}. For more background on Dirac structures, see e.g. \cite{bursztyn}. We also mention here that the recent work \cite{coregular} addresses coregular submanifolds $X\subset(M,\pi)$ for which additionally  
	$TX^{\perp_{\pi}}\cap TX$ is trivial, ensuring that $X$ has an induced Poisson structure. We warn the reader that these submanifolds are also referred to as ``coregular'' in \cite{coregular}. In the book \cite{bookpoiss}, these are called ``coregular Poisson-Dirac''.
	
	\begin{ack}
		This work was supported by EOS project G0H4518N, funded by Fonds Wetenschappelijk Onderzoek (FWO) and Fonds de la Recherche Scientifique (FNRS). I would like to thank my supervisor Ioan M\u{a}rcu\c{t} for very useful suggestions, and the Max Planck Institute for Mathematics in Bonn for its hospitality and financial support.
	\end{ack}
	
	\section{The saturation of a coregular submanifold}\label{sec:first}
	
	
	In this section, we discuss the saturation of embedded submanifolds $X$ in a Poisson manifold $(M,\pi)$. Our aim is to give sufficient conditions on $X$ that ensure smoothness of its saturation in a neighborhood of $X$. We introduce the class of \emph{coregular} submanifolds $X\subset (M,\pi)$, and we show that such a submanifold $X$ has a neighborhood $U$ in $M$ such that the saturation of $X$ in $(U,\pi|_{U})$ is an embedded Poisson submanifold.
	
	\begin{defi}
		The \textbf{saturation} of a submanifold $X$ of a Poisson manifold $(M,\pi)$ is the union of all the leaves of $(M,\pi)$ that intersect $X$. We denote the saturation of $X$ by $Sat(X)$.
	\end{defi}
	
	Recall that a Poisson submanifold $P\subset(M,\pi)$ is said to be complete if the inclusion  $(P,\pi_{P})\hookrightarrow(M,\pi)$ is a complete Poisson map \cite[\S 6.2]{models}.
	It is clear that, given a submanifold $X\subset(M,\pi)$, the saturation $Sat(X)$ is the smallest complete Poisson submanifold of $(M,\pi)$ containing $X$, provided it is smooth. Indeed, a complete Poisson submanifold $P\subset (M,\pi)$ is saturated \cite[Prop. 6.1]{models}, so if $X\subset P$ then $Sat(X)\subset Sat(P)=P$.

	The saturation of a submanifold can be very wild; in general it does not have a submanifold structure. For instance, consider the $x$-axis in the log-symplectic manifold $(\RR^{2},x\partial_{x}\wedge\partial_{y})$; its saturation is $\{x<0\}\cup\{(0,0)\}\cup\{x>0\}$. Clearly, this saturation doesn't even contain a Poisson submanifold around the $x$-axis.

	
	We now single out classes of submanifolds $X\subset(M,\pi)$ that do satisfy this property, i.e. whose saturation contains a Poisson submanifold which contains $X$. Examples of such submanifolds are transversals (whose saturation is open, and therefore a Poisson submanifold) and Poisson submanifolds. These are extreme cases of what we call \emph{coregular} submanifolds.

	\begin{defi}
		Given a Poisson manifold $(M,\pi)$, we call an embedded submanifold $X\subset M$ \textbf{coregular} if the map $\mathrm{pr}\circ\pi^{\sharp}:T^{*}M|_{X}\rightarrow TM|_{X}/TX$ has constant rank. 
	\end{defi}
	
	Note indeed that transversals and Poisson submanifolds are exactly those submanifolds $X\subset (M,\pi)$ for which the map $\mathrm{pr}\circ\pi^{\sharp}$ is respectively of full rank or identically zero. 
	
	We will now list some more observations about coregular submanifolds. For any submanifold $X\subset(M,\pi)$, we denote its $\pi$-orthogonal by $TX^{\perp_{\pi}}:=\pi^{\sharp}(TX^{\circ})$. If $x\in X$ and $L$ is the symplectic leaf through $x$, then $T_{x}X^{\perp_{\pi}}$ is the symplectic orthogonal of $T_{x}X\cap T_{x}L$ in the symplectic vector space $\left(T_{x}L,(\pi|_{L})^{-1}_{x}\right)$. Various types of submanifolds in Poisson geometry are defined in terms of their $\pi$-orthogonal; see \cite{bookpoiss} and \cite{Zambonsubmanifolds} for a systematic overview.
	\begin{enumerate}[a)]
		\item  Given a submanifold $X\subset(M,\pi)$, we get an exact sequence at points $x\in X$:
		\begin{equation}\label{sequence}
			\begin{tikzcd}[row sep=scriptsize, column sep=scriptsize]
				&0\arrow{r} &\big(T_{x}X^{\perp_{\pi}}\big)^{\circ}\arrow[r,hook] &T_{x}^{*}M\arrow[r,"\mathrm{pr}\circ\pi^{\sharp}"] &T_{x}M/T_{x}X.
			\end{tikzcd}
		\end{equation}
		Hence, $X\subset(M,\pi)$ is coregular exactly when $TX^{\perp_{\pi}}$ has constant rank. This observation explains the name ``coregular'' because, as mentioned above, $TX^{\perp_{\pi}}$ consists of the symplectic orthogonals to the (singular) distribution $TX\cap\pi^{\sharp}(T^{*}M|_{X})$.
		\item We give an alternative characterization of coregular submanifolds $X\subset (M,\pi)$ in Dirac geometric terms. Denote by $L_{\pi}$ the Dirac structure $L_{\pi}=\{\pi^{\sharp}(\alpha)+\alpha:\alpha\in T^{*}M\}$ defined by $\pi$, and let $i:X\hookrightarrow (M,\pi)$ be any submanifold. Then $X$ is coregular exactly when $L_{\pi}\cap\ker((\mathrm{d}i)^{*})$ has constant rank. Indeed, $L_{\pi}\cap\ker((\mathrm{d}i)^{*})=\ker(\pi^{\sharp})\cap TX^{\circ}$, so for any $x\in X$, we have
		\[
		\dim\left(L_{\pi}\cap\ker((\mathrm{d}i)^{*})\right)_{x}=\dim\big(T_{x}X^{\circ}\big)-\dim\big(T_{x}X^{\perp_{\pi}}\big).
		\]
		Hence, coregular submanifolds are exactly those  $X\subset (M,\pi)$ for which the Dirac structure $L_{\pi}$ automatically pulls back to a smooth Dirac structure on $X$ \cite[Prop. 1.10]{bursztyn}.
	\end{enumerate}

	
	We mention here that $X$ being coregular is not a necessary condition for $Sat(X)$ to contain a Poisson submanifold around $X$, as illustrated by the following examples.
	
	\begin{exs}
		\begin{enumerate}[i)]
			\item Take the Lie-Poisson structure $\big(\mathfrak{so}(3)^{*}, z\partial_{x}\wedge\partial_{y}+x\partial_{y}\wedge\partial_{z}+y\partial_{z}\wedge\partial_{x}\big)$ and let $X$ be the plane defined by $z=0$. The symplectic foliation of $\mathfrak{so}(3)^{*}$ consists of concentric spheres of radius $r\geq 0$ centered at the origin, so that $Sat(X)=\mathfrak{so}(3)^{*}$. However, $TX^{\perp_{\pi}}=\text{Span}\{y\partial_{x}-x\partial_{y}\}$ vanishes at the origin, so $X$ is not coregular.
			\item Consider the regular Poisson manifold $(\RR^{3},\partial_{x}\wedge\partial_{y})$ and let $X$ be defined by the equation $z=x^{3}$. Then the saturation $Sat(X)$ is all of $\mathbb{R}^{3}$, but $X$ is not coregular. Indeed, 
			we have $TX^{\perp_{\pi}}=\text{Span}\{-3x^{2}\partial_{y}\}$, which drops rank at points of the form $(0,y,0)\in X$.
		\end{enumerate}
	\end{exs}
	
	To construct a Poisson submanifold around the coregular submanifold $X\subset (M,\pi)$, we use some notions from contravariant geometry and some theory of dual pairs \cite{realizations},\cite{transversals}. 
	
	\begin{defi}\label{spray}
		A \textbf{Poisson spray} on a Poisson manifold $(M,\pi)$ is a vector field $\chi$ on the cotangent bundle $T^{*}M$ satisfying:
		\begin{enumerate}[i)]
			\item $\mathrm{d}\mathrm{pr}(\chi(\xi))=\pi^{\sharp}(\xi)$ for all $\xi\in T^{*}M$,
			\item $m_{t}^{*}\chi=t\chi$ for all $t>0$,
		\end{enumerate}
		where $\mathrm{pr}:T^{*}M\rightarrow M$ is the projection and $m_{t}:T^{*}M\rightarrow T^{*}M$ is multiplication by $t$.
	\end{defi}
	
	Poisson sprays $\chi\in\mathfrak{X}(T^{*}M)$ exist on any Poisson manifold. Since $\chi$ vanishes along the zero section $M\subset T^{*}M$, there exists a neighborhood $\Sigma\subset T^{*}M$ of $M$ on which the flow $\phi_{\chi}^{t}$ is defined for all times $t\in[0,1]$. One can then define the \textbf{spray exponential} $\exp_{\chi}$ of $\chi$ by
	\[
	\exp_{\chi}:\Sigma\subset T^{*}M\rightarrow M:\xi\mapsto \mathrm{pr}(\phi_{\chi}^{1}(\xi)).
	\]
	
	This neighborhood $\Sigma\subset T^{*}M$ also supports a closed two-form $\Omega_{\chi}$, which is defined by averaging the canonical symplectic form $\omega_{can}$ with respect to the flow $\phi_{\chi}^{t}$ of the Poisson spray $\chi\in\mathfrak{X}(T^{*}M)$:
	\[
	\Omega_{\chi}:=\int_{0}^{1}\big(\phi_{\chi}^{t}\big)^{*}\omega_{can}\mathrm{d}t.
	\]
	It was proved in \cite{realizations} that $\Omega_{\chi}$ is non-degenerate along the zero zection $M\subset T^{*}M$, so shrinking $\Sigma\subset T^{*}M$ if necessary, we can assume that $\Omega_{\chi}$ is symplectic on $\Sigma$. By \cite[Lemma 25] {transversals}, the symplectic manifold $(\Sigma,\Omega_{\chi})$ fits in a \textbf{full dual pair}
	\begin{equation}\label{pair}
		\begin{tikzcd}
			(M,\pi)&(\Sigma,\Omega_{\chi})\arrow{r}{\exp_{\chi}}\arrow[l,"\mathrm{pr}",swap]&(M,-\pi).
		\end{tikzcd}
	\end{equation}
	That is, denoting by $\pi_{\chi}:=\Omega_{\chi}^{-1}$ the Poisson structure corresponding with $\Omega_{\chi}$, the maps $\mathrm{pr}:\big(\Sigma,\pi_{\chi}\big)\rightarrow(M,\pi)$ and $\exp_{\chi}:\big(\Sigma,\pi_{\chi}\big)\rightarrow (M,-\pi)$ are surjective Poisson submersions with symplectically orthogonal fibers: $(\ker \mathrm{d}\mathrm{pr})^{\perp_{\Omega_{\chi}}}=\ker{\mathrm{d}\exp_{\chi}}$.
	
	Both legs in the diagram \eqref{pair} are symplectic realizations. We will need the following lemma, which concerns the interplay between symplectic realizations and coregular submanifolds of a Poisson manifold.
	
	\begin{lemma}\label{realization}
		Let $X\subset (M,\pi)$ be a coregular submanifold and let $\mu:(\Sigma,\Omega)\rightarrow (M,\pi)$ be a symplectic realization. Then $(\ker \mathrm{d}\mu)^{\perp_{\Omega}}\cap T(\mu^{-1}(X))$ has constant rank, equal to the corank of $TX^{\perp_{\pi}}\subset TM|_{X}$.
	\end{lemma}
	\begin{proof}
		Denote by $\pi_{\Omega}:=\Omega^{-1}$ the Poisson structure corresponding with $\Omega$. First note that, for $\xi\in\mu^{-1}(X)$, we have
		\[
		(\ker \mathrm{d}\mu)_{\xi}^{\perp_{\Omega}}
		=\pi_{\Omega}^{\sharp}\big((\mathrm{d}\mu)_{\xi}^{*}T_{\mu(\xi)}^{*}M\big).
		\]
		Since for any $\beta\in T_{\mu(\xi)}^{*}M$, we have that $(\mathrm{d}\mu)_{\xi}\pi_{\Omega}^{\sharp}\big((\mathrm{d}\mu)_{\xi}^{*}\beta\big)=\pi^{\sharp}(\beta)$ belongs to $T_{\mu(\xi)}X$ exactly when $\beta\in (T_{\mu(\xi)}X^{\perp_{\pi}})^{\circ}$, we obtain
		\begin{align}\label{T}
			(\ker \mathrm{d}\mu)_{\xi}^{\perp_{\Omega}}\cap T_{\xi}(\mu^{-1}(X))&=(\ker \mathrm{d}\mu)_{\xi}^{\perp_{\Omega}}\cap(\mathrm{d}\mu)_{\xi}^{-1}(T_{\mu(\xi)}X)\nonumber\\
			&=\pi_{\Omega}^{\sharp}\left((\mathrm{d}\mu)_{\xi}^{*}\big(T_{\mu(\xi)}X^{\perp_{\pi}}\big)^{\circ}\right).
		\end{align}
		Hence, the rank of $(\ker \mathrm{d}\mu)^{\perp_{\Omega}}\cap T(\mu^{-1}(X))$ is constant, equal to $\dim M - rk(TX^{\perp_{\pi}})$.
	\end{proof}
	
	We now prove that for a coregular submanifold $X\subset (M,\pi)$, there exists an embedded Poisson submanifold of $(M,\pi)$ containing $X$ that lies in the saturation $Sat(X)$. This Poisson submanifold is in fact the saturation of $X$ in a neighborhood $(U,\pi|_{U})$ of $X$.
	
	\begin{thm}\label{satsmooth}
		Let $X\subset (M,\pi)$ be a coregular submanifold. 
		\begin{enumerate}
			\item There exists an embedded Poisson submanifold $(P,\pi_{P})\subset(M,\pi)$ containing $X$ that lies inside the saturation $Sat(X)$.
			\item Shrinking $P$ if necessary, there exists a neighborhood $U$ of $X$ in $M$ such that $(P,\pi_{P})$ is the saturation of $X$ in $(U,\pi|_{U})$.
		\end{enumerate}
	\end{thm}
	\begin{proof}
		We divide the proof into four steps.
		
		\vspace{0.2cm}
		\noindent
		\underline{Step 1:} Construction of the embedded submanifold $P\subset M$.
		
		\vspace{0.1cm}
		\noindent
		Choose a Poisson spray $\chi\in\mathfrak{X}(T^{*}M)$ and denote by $\exp_{\chi}:\Sigma\subset T^{*}M\rightarrow M$ the corresponding spray exponential. Notice that the restriction $\exp_{\chi}:\Sigma|_{X}\rightarrow M$ takes values in $Sat(X)$, because sprays trace cotangent paths \cite[\S 1]{Apaths}. Indeed, for $\xi\in\Sigma|_{X}$, the curve $\gamma(t):=\mathrm{pr}(\phi_{\chi}^{t}(\xi))$ satisfies  
			\begin{align*}
				\gamma'(t)&=(\mathrm{d}\mathrm{pr})_{\phi_{\chi}^{t}(\xi)}\left(\frac{\mathrm{d}}{\mathrm{d}t}\phi_{\chi}^{t}(\xi)\right)
				=(\mathrm{d}\mathrm{pr})_{\phi_{\chi}^{t}(\xi)}\left(\chi(\phi_{\chi}^{t}(\xi))\right)
				=\pi^{\sharp}_{\gamma(t)}(\phi_{\chi}^{t}(\xi)),
			\end{align*}
			showing that $t\mapsto(\phi_{\chi}^{t}(\xi),\gamma(t))$ is a cotangent path between $\gamma(0)=\mathrm{pr}(\xi)$ and $\gamma(1)=\exp_{\chi}(\xi)$. Consequently, $\mathrm{pr}(\xi)$ and $\exp_{\chi}(\xi)$ lie in the same leaf of $(M,\pi)$, hence $\exp_{\chi}(\Sigma|_{X})\subset Sat(X)$.
		
		Choosing a complement to $TX^{\perp_{\pi}}$ in $TM|_{X}$, we get an inclusion $(TX^{\perp_{\pi}})^{*}\subset T^{*}M|_{X}$. The restriction of the spray exponential $\exp_{\chi}:(TX^{\perp_{\pi}})^{*}\cap\Sigma\rightarrow M$ fixes points of $X$, and its differential along $X$ reads \cite[Lemma 24]{transversals}:
		\[
		\mathrm{d}\exp_{\chi}:T_{x}X\oplus(T_{x}X^{\perp_{\pi}})^{*}\rightarrow T_{x}M:(v,\xi)\mapsto v+\pi^{\sharp}(\xi).
		\]
		This map is injective, for if $\pi^{\sharp}(\xi)=-v\in T_{x}X$, then $\xi\in(\pi^{\sharp})^{-1}(T_{x}X)=(T_{x}X^{\perp_{\pi}})^{\circ}$ and therefore $\xi\in (T_{x}X^{\perp_{\pi}})^{*}\cap(T_{x}X^{\perp_{\pi}})^{\circ}=\{0\}$.  Proposition \ref{embedding} and Remark \ref{nbhd} in the Appendix imply that, shrinking $\Sigma$ if necessary,  the map $\exp_{\chi}:(TX^{\perp_{\pi}})^{*}\cap\Sigma\rightarrow M$ is an embedding.
		Setting
		\[
		P:=\exp_{\chi}\big((TX^{\perp_{\pi}})^{*}\cap\Sigma\big),
		\] 
		this is an embedded submanifold of $M$ containing $X$ that lies inside $Sat(X)$.
		
		\vspace{0.2cm}
		\noindent
		\underline{Step 2:} Shrinking $\Sigma$ if necessary, we have $P=\exp_{\chi}(\Sigma|_{X})$.
		
		\vspace{0.1cm}
		\noindent
		To see this, let us denote for short $\Sigma_{X}:=\Sigma|_{X}\subset T^{*}M|_{X}$ and $\widetilde{\Sigma_{X}}:=(TX^{\perp_{\pi}})^{*}\cap\Sigma$.
		
		First, we claim that the restriction $\left.\exp_{\chi}\right|_{\Sigma_{X}}$ has constant rank, equal to the rank of $\left.\exp_{\chi}\right|_{\widetilde{\Sigma_{X}}}$. Indeed, using the self-dual pair \eqref{pair}, we have that
		\[
		\ker\left(\mathrm{d}\left(\exp_{\chi}|_{\Sigma_{X}}\right)\right)=\ker(\mathrm{d}\exp_{\chi})\cap T(\mathrm{pr}^{-1}(X))=\ker(\mathrm{d}\mathrm{pr})^{\perp_{\Omega_{\chi}}}\cap T(\mathrm{pr}^{-1}(X)),
		\]
		which has constant rank equal to $\dim M- rk(TX^{\perp_{\pi}})$ by Lemma \ref{realization}. Consequently, the rank of $\left.\exp_{\chi}\right|_{\Sigma_{X}}$ is equal to $\dim X + rk(TX^{\perp_{\pi}})$, which is the rank of $\left.\exp_{\chi}\right|_{\widetilde{\Sigma_{X}}}$.

		Using the claim just proved, we will now show that $\exp_{\chi}(\widetilde{\Sigma_{X}})=\exp_{\chi}(\Sigma_{X})$, shrinking $\Sigma$ if necessary. It is enough to prove that every point $\xi\in\widetilde{\Sigma_{X}}$ has a neighborhood $V^{\xi}\subset\Sigma_{X}$ such that $\exp_{\chi}(V^{\xi})\subset\exp_{\chi}(\widetilde{\Sigma_{X}})$. We keep in mind the diagram
		\[
		\begin{tikzcd}[column sep=large, row sep=large]
			\widetilde{\Sigma_{X}}\subset\big(TX^{\perp_{\pi}}\big)^{*}\arrow[r,hookrightarrow]\arrow[d,"\exp_{\chi}|_{\widetilde{\Sigma_{X}}}",swap,"\simeq"'{rotate=90,yshift=-5pt,xshift=-6pt}] & \Sigma_{X}\subset T^{*}M|_{X}\arrow[d,"\exp_{\chi}|_{\Sigma_{X}}"]\\
			\exp_{\chi}(\widetilde{\Sigma_{X}})\arrow[r,hookrightarrow] & M
		\end{tikzcd}.
		\]
		Pick $\xi\in\widetilde{\Sigma_{X}}\subset\Sigma_{X}$. Since $\exp_{\chi}|_{\Sigma_{X}}$ has constant rank, there is an open  $U^{\xi}\subset\Sigma_{X}$ around $\xi$ such that $\exp_{\chi}(U^{\xi})\subset M$ is an embedded submanifold. As $\exp_{\chi}$ is an embedding on $\widetilde{\Sigma_{X}}$, also $\exp_{\chi}(U^{\xi}\cap\widetilde{\Sigma_{X}})\subset M$ is an embedded submanifold. Since $\dim\exp_{\chi}(U^{\xi}\cap\widetilde{\Sigma_{X}})=\dim\exp_{\chi}(U^{\xi})$ by the previous claim, the inverse function theorem implies that $\exp_{\chi}(U^{\xi}\cap\widetilde{\Sigma_{X}})$ is open in $\exp_{\chi}(U^{\xi})$. Since $\exp_{\chi}|_{U^{\xi}}:U^{\xi}\rightarrow\exp_{\chi}(U^{\xi})$ is continuous, the set $\exp_{\chi}|_{U^{\xi}}^{-1}\big(\exp_{\chi}(U^{\xi}\cap\widetilde{\Sigma_{X}})\big)$ is open in $U^{\xi}$, hence in $\Sigma_{X}$. Setting $V^{\xi}:=\exp_{\chi}|_{U^{\xi}}^{-1}\big(\exp_{\chi}(U^{\xi}\cap\widetilde{\Sigma_{X}})\big)$ proves the assertion.
		
		\vspace{0.2cm}
		\noindent
		\underline{Step 3:} $P$ is a Poisson submanifold of $(M,\pi)$.
		
		\vspace{0.1cm}
		\noindent
		We use the previous step, which states that $P=\exp_{\chi}(\Sigma_{X})$. Pick a point $x\in P$ and let $\xi\in\Sigma_{X}$ be such that $\exp_{\chi}(\xi)=x$. We have to show that $\pi^{\sharp}(T_{x}P^{\circ})=\{0\}$. Making use of the dual pair \eqref{pair}, we have
		\[
		\pi^{\sharp}(T_{x}P^{\circ})=\left[(\mathrm{d}\exp_{\chi})_{\xi}\circ\pi_{\chi}^{\sharp}\circ(\mathrm{d}\exp_{\chi})_{\xi}^{*}\right]\big(T_{x}P^{\circ}\big),
		\]
		so it is enough to show that
		\[
		\pi_{\chi}^{\sharp}\left((\mathrm{d}\exp_{\chi})_{\xi}^{*}(T_{x}P^{\circ})\right)\subset \ker(\mathrm{d}\exp_{\chi})_{\xi}=\pi_{\chi}^{\sharp}(\ker \mathrm{d}\mathrm{pr})_{\xi}^{\circ}.
		\]
		To see that this inclusion holds, note that $(\ker \mathrm{d}\mathrm{pr})_{\xi}\subset T_{\xi}\Sigma_{X}$ and $(\mathrm{d}\exp_{\chi})_{\xi}(T_{\xi}\Sigma_{X})\subset T_{x}P$, which implies that
		\[
		(\mathrm{d}\exp_{\chi})_{\xi}^{*}(T_{x}P^{\circ})\subset (T_{\xi}\Sigma_{X})^{\circ}\subset(\ker \mathrm{d}\mathrm{pr})_{\xi}^{\circ}.
		\]
		We now showed that $P\subset(M,\pi)$ is a Poisson submanifold, which finishes Step 3.
		
		\vspace{0.2cm}
		\noindent
		\underline{Step 4:} Construction of the neighborhood $U$ of $X$.
		
		\vspace{0.1cm}
		\noindent
		The idea is to extend $\exp_{\chi}:(TX^{\perp_{\pi}})^{*}\cap\Sigma\rightarrow M$ to a local diffeomorphism, using the same reasoning as in the proof of Proposition \ref{embedding} in the Appendix. Choosing a complement 
		\[
		TM|_{X}=TX\oplus\pi^{\sharp}\big(TX^{\perp_{\pi}}\big)^{*}\oplus C=TP|_{X}\oplus C,
		\]
		and a linear connection $\nabla$ on $TM$, we obtain a map
		\[
		\psi:V\subset \big((TX^{\perp_{\pi}})^{*}\oplus C\big)\rightarrow M:(\xi,c)\mapsto\exp_{\nabla}\big(Tr_{\exp_{\chi}(t\xi)}c\big),
		\]
		which is a diffeomorphism onto a neighborhood of $X$.
		Here $V$ is a suitable convex neighborhood of the zero section, and $Tr_{\exp_{\chi}(t\xi)}$ denotes parallel transport along the curve $t\mapsto\exp_{\chi}(t\xi)$ for $t\in[0,1]$. Note that $\psi$ satisfies $\psi(\xi,0)=\exp_{\chi}(\xi)$. Consequently, shrinking $P$ if necessary, we can assume that 
		\[
		P=\psi\left(V\cap\left(\big(TX^{\perp_{\pi}}\big)^{*}\oplus\{0\}\right)\right).
		\]
		We now set $U:=\psi(V)$, and we check that $P$ is the Poisson saturation of $X$ in $(U,\pi|_{U})$.
		
		On one hand, since $(TX^{\perp_{\pi}})^{*}$ is closed in $(TX^{\perp_{\pi}})^{*}\oplus C$, also $P$ is closed in $U$. It follows that $P$ is saturated, being a properly embedded Poisson submanifold. Hence, the saturation of $X$ in $(U,\pi|_{U})$ is contained in $P$. On the other hand, if $\exp_{\chi}(\xi)=\psi(\xi,0)\in P\subset U$, then also $\exp_{\chi}(t\xi)\in U$ for $t\in[0,1]$ since $V$ is convex. Consequently, the path $t\mapsto(\phi^{t}_{\chi}(\xi),\exp_{\chi}(t\xi))$ is a cotangent path covering a path in $U$ that connects $\exp_{\chi}(\xi)$ with a point in $X$. This shows that $\exp_{\chi}(\xi)$ is contained in the Poisson saturation of $X$ in $(U,\pi|_{U})$.
	\end{proof}
	
	
	The theorem above shows that the saturation of a coregular submanifold $X\subset(M,\pi)$ in some neighborhood $(U,\pi|_{U})$ of $X$ is an embedded Poisson submanifold. Clearly, one cannot take $U$ to be all of $M$ in general. In this respect, we have the following sufficient condition.
	\begin{cor}\label{open}
		Let $X\subset(M,\pi)$ be a coregular submanifold. If the submanifold $P$ constructed in Theorem \ref{satsmooth} is open in $Sat(X)$ for the induced topology, then $Sat(X)$ is an embedded submanifold of $M$.
	\end{cor}
	\begin{proof}
		Recall the following general fact \cite[Lemma 4.5]{bookpoiss}: if $\{N_i\}_{i\in\mathcal{I}}$ is a collection of embedded submanifolds of $M$, all of the same dimension, such that $N_i\cap N_j$ is open in $N_i$ for all $i,j\in\mathcal{I}$, then $N:=\cup_{i\in\mathcal{I}}N_i$ is a smooth manifold, possibly not second countable, for which the inclusion $N\hookrightarrow M$ is an immersion. The smooth structure is uniquely determined by the condition that the maps $N_i\hookrightarrow N$ are smooth open embeddings.	
		
		We want to apply this fact to the collection $\{\phi^{1}_{X_{f}}(P): f\in C^{\infty}_{c}([0,1]\times M)\}$, where $\phi^{1}_{X_{f}}$ denotes the time $1$-flow of the Hamiltonian vector field associated with the compactly supported function $f\in C^{\infty}_{c}([0,1]\times M)$. We have to check that $\phi^{1}_{X_{f}}(P)\cap\phi^{1}_{X_{g}}(P)$ is open in $\phi^{1}_{X_{f}}(P)$. To this end, note that both $\phi^{1}_{X_{f}}(P)$ and $\phi^{1}_{X_{g}}(P)$ are open in $Sat(X)$, since by assumption $P$ is open in $Sat(X)$ and $\phi^{1}_{X_{f}},\phi^{1}_{X_{g}}$ are diffeomorphisms preserving $Sat(X)$. Hence, also $\phi^{1}_{X_{f}}(P)\cap\phi^{1}_{X_{g}}(P)$ is open in $Sat(X)$, so there exists an open $V\subset M$ such that
		\[
		\phi^{1}_{X_{f}}(P)\cap\phi^{1}_{X_{g}}(P)=V\cap Sat(X).
		\]
		Since also $\phi^{1}_{X_{f}}(P)=U\cap Sat(X)$ for some open $U\subset M$, we obtain
		\[
		\phi^{1}_{X_{f}}(P)\cap\phi^{1}_{X_{g}}(P)=\phi^{1}_{X_{f}}(P)\cap\phi^{1}_{X_{g}}(P)\cap U=V\cap(U\cap Sat(X))=V\cap\phi^{1}_{X_{f}}(P),
		\]
		which shows that $\phi^{1}_{X_{f}}(P)\cap\phi^{1}_{X_{g}}(P)$ is open in $\phi^{1}_{X_{f}}(P)$.
		
		So we can apply the fact mentioned above, which gives $Sat(X)$ a smooth manifold structure, a priori not necessarily second countable, for which $Sat(X)\hookrightarrow M$ is an immersion. But since the topology of this smooth structure is generated by open subsets of the submanifolds $\phi^{1}_{X_{f}}(P)$, it coincides with the induced topology on $Sat(X)$. In particular, it is second countable, and $Sat(X)$ is an embedded submanifold of $M$.
	\end{proof}
	
	The proof of Corollary \ref{open} breaks down if $P$ is not open in $Sat(X)$, because the technical requirements of \cite[Lemma 4.5]{bookpoiss} are no longer met. See for instance Example \ref{fail} below.
	
	\begin{remark}\label{rem:fail}
		We comment on the condition in Corollary \ref{open} that $P=\exp_{\chi}(\Sigma|_{X})$ needs to be open in $Sat(X)$ for the induced topology. 
		This occurs exactly when we are able to find a small transversal $\tau\subset(M,\pi)$ to the leaves such that $\tau\cap Sat(X)=X$. 
		
		To see that then $\exp_{\chi}(\Sigma|_{X})$ is indeed open in $Sat(X)$ with respect to the induced topology, we note that $\exp_{\chi}:\Sigma|_{\tau}\rightarrow M$ is a submersion, shrinking $\Sigma$ if necessary. Indeed, at points $p\in\tau$, the differential
		\[
		\mathrm{d}\exp_{\chi}:T_{p}\tau\oplus T_{p}^{*}M\rightarrow T_{p}M:(v,\xi)\mapsto v+\pi^{\sharp}(\xi)
		\]
		is surjective since $\tau\subset(M,\pi)$ is a transversal. Hence $\exp_{\chi}$ is of maximal rank in a neighborhood of $\tau\subset\Sigma|_{\tau}$. In particular, shrinking $\Sigma$ if needed, we have that $\exp_{\chi}\big(\Sigma|_{\tau}\big)\subset M$ is open. It now suffices to remark that $\exp_{\chi}(\Sigma|_{X})=\exp_{\chi}(\Sigma|_{\tau})\cap Sat(X)$. The forward inclusion is clear, since $X\subset\tau$ and $\exp_{\chi}(\Sigma|_{X})\subset Sat(X)$. For the backward inclusion, assume that $(p,\xi)\in\Sigma|_{\tau}$ is such that $\exp_{\chi}(\xi)\in Sat(X)$. Since $p$ lies in the same leaf as $\exp_{\chi}(\xi)\in Sat(X)$ and $Sat(X)$ is saturated, it follows that $p\in Sat(X)$. Consequently, $p\in \tau\cap Sat(X)=X$. This shows that $\exp_{\chi}(\Sigma|_{X})=\exp_{\chi}(\Sigma|_{\tau})\cap Sat(X)$ is open in $Sat(X)$ for the induced topology.
	\end{remark}
	
	In the particular case where $X$ is a point, then $Sat(X)$ is just the leaf through $X$, which is well-known to possess a natural smooth structure. Indeed, each leaf of a Poisson manifold is an initial submanifold, so in particular it possesses a unique smooth structure that turns it into an immersed submanifold. For an arbitrary coregular submanifold $X$, its saturation does not have a natural smooth structure, as illustrated in the following example.
	

	\begin{ex}\label{fail}
		We look at the manifold $(\mathbb{R}^{3}\times S^{1},x,y,z,\theta)$ with Poisson structure $\pi=\partial_{z}\wedge\partial_{\theta}$. Consider the curve $\beta:\mathbb{R}\rightarrow\mathbb{R}^{3}:t\mapsto(\sin(2t),\sin(t),t)$, which is a ``figure eight'' coming out of the $xy$-plane. Denote its image by $\mathcal{C}\subset\mathbb{R}^{3}$, and let $\mathcal{C}_{base}$ be the projection of $\mathcal{C}$ onto the $xy$-plane. The submanifold $X:=\mathcal{C}\times S^{1}\subset\mathbb{R}^{3}\times S^{1}$ is embedded, and we claim that it is coregular. 
		To see this, we only have to check that $\dim(T_{p}X\cap T_{p}L)$ is constant for $p\in X$, where $L$ denotes the leaf through $p$. Since at a point $p=(\beta(t_0),\theta_0)$ we have
		\[
		T_{p}X=\text{Span}\{\left.\partial_{\theta}\right|_{p},2\cos(2t_0)\left.\partial_{x}\right|_{p}+\cos(t_0)\left.\partial_{y}\right|_{p}+\left.\partial_{z}\right|_{p}\},
		\]
		it is clear that $T_{p}X\cap T_{p}L=\text{Span}\{\left.\partial_{\theta}\right|_{p}\}$, since $\cos(t_0)$ and $\cos(2t_0)$ cannot both be zero. Hence, $X\subset(\mathbb{R}^{3}\times S^{1},\pi)$ is coregular. Its saturation is given by $Sat(X)=\mathcal{C}_{base}\times\mathbb{R}\times S^{1}$.
		
		The saturation has two obvious smooth structures that turn it into an immersed submanifold, coming from those on the ``figure eight''. But neither of them can be called natural, because for both smooth structures the inclusion $X\hookrightarrow Sat(X)$ is not even continuous.
		
		
		Let us also come back to the proof of Corollary \ref{open} and see why it fails in this case. We refer to the figure below, where we removed the $S^{1}$-factor, which is not essential to the spirit of the example. The embedded submanifold $P$ in this case is obtained by slightly thickening the curve in vertical direction. One can take a Hamiltonian flow $\phi^{1}_{X_f}$ such that $\phi^{1}_{X_f}(P)\cap P$ consists of vertical segments of the line in which the cylinder intersects itself. This is not an open subset of $P$, so we can no longer apply \cite[Lemma 4.5]{bookpoiss}. And indeed, the conclusion of Corollary \ref{open} fails in this example.
		
		\begin{figure}[H]
			\includegraphics[height=5cm,width=3cm, trim=4cm 2cm 4cm 0.8cm]{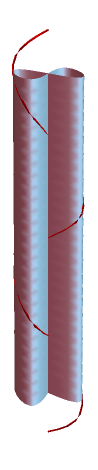}
			\caption{The coregular submanifold $X$ and its saturation $Sat(X)$. This is the picture in $\mathbb{R}^{3}$; the $S^{1}$-factor is omitted for the sake of depiction.}
		\end{figure}
	\end{ex}

	As a consequence of Theorem \ref{satsmooth}, we obtain an alternative characterization of coregular submanifolds. It turns out that the two extreme examples -- Poisson submanifolds and transversals -- are the building blocks of any coregular submanifold.
	
	\begin{prop}\label{intersection}
		A submanifold $X\subset(M,\pi)$ is coregular if and only if $X$ is the intersection of a Poisson submanifold $P\subset(M,\pi)$ with a transversal $\tau\subset(M,\pi)$.
	\end{prop}
	A transversal $\tau\subset(M,\pi)$ is also transverse to any Poisson submanifold $P\subset (M,\pi)$, since the intersection of $P$ with any leaf of $(M,\pi)$ is open in the leaf. Indeed, if $p\in P$ and $L$ is the leaf through $p$, then 
	\[
	T_{p}M=T_{p}\tau+T_{p}L=T_{p}\tau+T_{p}(P\cap L)\subset T_{p}\tau + T_{p}P,
	\]
	which shows that $\tau\pitchfork P$. In particular, the intersection $\tau\cap P$ is smooth.
	\begin{proof}[Proof of Prop. \ref{intersection}]
		First assume that $X\subset (M,\pi)$ is a coregular submanifold. Theorem \ref{satsmooth} then gives a Poisson submanifold $P\subset(M,\pi)$ containing $X$, and the proof shows that
		\begin{equation}\label{normal}
			TP|_{X}=TX\oplus\pi^{\sharp}\big(TX^{\perp_{\pi}}\big)^{*}. 
		\end{equation}
		Choose a complement $E$ to this subbundle of $TM|_{X}$, i.e. $TM|_{X}=TX\oplus\pi^{\sharp}\big(TX^{\perp_{\pi}}\big)^{*}\oplus E$. Using a (metric) exponential map, we can construct a submanifold $\tau\subset M$ containing $X$ such that $T\tau|_{X}=TX\oplus E$.
		For small enough $\tau$, we have $\tau\cap P=X$, and moreover
		\[
		TM|_{X}=TX\oplus\pi^{\sharp}\big(TX^{\perp_{\pi}}\big)^{*}\oplus E=\big(TX+\text{Im}(\pi^{\sharp}|_{X})\big)\oplus E=\text{Im}(\pi^{\sharp}|_{X})+T\tau|_{X},
		\]
		which shows that $\tau$ is a transversal along $X$. Shrinking $\tau$ if necessary, this implies that $\tau$ is a transversal in $(M,\pi)$. This proves the forward implication.
		
		For the converse, assume that $X=\tau\cap P$ is a submanifold of $M$, where $P\subset (M,\pi)$ is a Poisson submanifold and $\tau\subset(M,\pi)$ is a transversal. Then $TX=T\tau|_{X}\cap TP|_{X}$, so that $TX^{\circ}=(T\tau|_{X})^{\circ}+(TP|_{X})^{\circ}$. Since $P$ is a Poisson submanifold, we get $TX^{\perp_{\pi}}=\pi^{\sharp}((T\tau|_{X})^{\circ})$. Since $\tau$ is a transversal, the restriction $\pi^{\sharp}|_{T\tau^{\circ}}$ is injective, hence $X$ is coregular.
	\end{proof}
	
	In what follows, we denote by $(P,\pi_P)$ the Poisson submanifold containing $X$ that was constructed in Theorem \ref{satsmooth}. We refer to $(P,\pi_P)$ as the \textbf{local Poisson saturation} of $X$. 
	In the next two sections, we prove a normal form theorem for $(P,\pi_P)$ around $X$.
	

	\section{The local model}\label{sec:second}
	
	This section introduces the local model for the local Poisson saturation $(P,\pi_P)$ of a coregular submanifold $X\subset (M,\pi)$. The local model is defined on the vector bundle $(TX^{\perp_{\pi}})^{*}$, which is indeed isomorphic with the normal bundle of $X$ in $P$. An explicit isomorphism is obtained by choosing an embedding $(TX^{\perp_{\pi}})^{*}\hookrightarrow T^{*}M|_{X}$ and then applying the bundle map $\pi^{\sharp}$, see equation \eqref{normal}. The local model involves some extra choices, which we now explain.
	
	Let $X\subset (M,\pi)$ be a coregular submanifold, and choose a \textbf{complement} $W$ to  $TX^{\perp_{\pi}}$ inside $TM|_{X}$. We obtain correspondingly an inclusion map $j:\big(TX^{\perp_{\pi}}\big)^{*}\hookrightarrow T^{*}M|_{X}$. Define skew-symmetric bilinear forms $\sigma\in\Gamma(\wedge^{2}TX^{\perp_{\pi}})$ and $\tau\in\Gamma(T^{*}X\otimes TX^{\perp_{\pi}})$ on the restricted tangent bundle $T\big((TX^{\perp_{\pi}})^{*}\big)|_{X}=TX\oplus(TX^{\perp_{\pi}})^{*}$ by the formulas
	\begin{align}\label{st}
		&\sigma(\xi_{1},\xi_{2})=\pi\big(j(\xi_{1}),j(\xi_{2})\big),\nonumber\\
		&\tau\big((v_{1},\xi_{1}),(v_{2},\xi_{2})\big)=\langle v_{1},j(\xi_{2})\rangle - \langle v_{2},j(\xi_{1})\rangle,
	\end{align}
	for $\xi_{1},\xi_{2}\in\big(T_{x}X^{\perp_{\pi}}\big)^{*}$ and $v_{1},v_{2}\in T_{x}X$. Denote by $\mathcal{E}_{W}(-\sigma-\tau)$ the set of all closed two-forms $\eta$, defined on a neighborhood of $X\subset(TX^{\perp_{\pi}})^{*}$, whose restriction to the zero section $X\subset(TX^{\perp_{\pi}})^{*}$ equals
	\begin{equation}\label{restr}
		\eta|_{X}=-\sigma\oplus -\tau\oplus 0 \in\Gamma(\wedge^{2}TX^{\perp_{\pi}})\oplus\Gamma(T^{*}X\otimes TX^{\perp_{\pi}})\oplus\Gamma(\wedge^{2}T^{*}X).
	\end{equation}
	We refer to a two-form $\eta\in\mathcal{E}_{W}(-\sigma-\tau)$ as a \textbf{closed extension} of $-\sigma-\tau$. Closed extensions of $-\sigma-\tau$ exist, see for instance \cite[Extension Theorem]{weinstein}. 
	
	The local model for the local Poisson saturation of a coregular submanifold $X\overset{i}{\hookrightarrow}(M,\pi)$ is now defined as follows: pull back the Dirac structure $i^{!}L_{\pi}$ on $X$ to $(TX^{\perp_{\pi}})^{*}$ under the submersion $\mathrm{pr}:(TX^{\perp_{\pi}})^{*}\rightarrow X$ and gauge transform by a closed extension $\eta\in\mathcal{E}_{W}(-\sigma-\tau)$. The obtained Dirac structure $\left(\mathrm{pr}^{!}(i^{!}L_{\pi})\right)^{\eta}$ indeed defines a Poisson stucture in a neighborhood of $X\subset(TX^{\perp_{\pi}})^{*}$, as we now show.

	\begin{prop}\label{model}
		Let $X\subset(M,\pi)$ be a coregular submanifold. Fix a complement $W$ to  $TX^{\perp_{\pi}}$ in $TM|_{X}$, define $\sigma\in\Gamma(\wedge^{2}TX^{\perp_{\pi}})$ and $\tau\in\Gamma(T^{*}X\otimes TX^{\perp_{\pi}})$ by the formulas \eqref{st} and
		let $\eta\in\mathcal{E}_{W}(-\sigma-\tau)$ be any closed extension. The Dirac structure $\left(\mathrm{pr}^{!}(i^{!}L_{\pi})\right)^{\eta}$ is Poisson on a neighborhood $U$ of $X\subset\big(TX^{\perp_{\pi}}\big)^{*}$.
	\end{prop}
	
	\begin{proof}
		It suffices to show that $\left(\mathrm{pr}^{!}(i^{!}L_{\pi})\right)^{\eta}$ is transverse to $T\big(TX^{\perp_{\pi}}\big)^{*}$ along $X$. By the sequence \eqref{sequence}, we have $i^{!}L_{\pi}=\big\{\pi^{\sharp}(\alpha)+(\mathrm{d}i)^{*}\alpha:\ \alpha\in(TX^{\perp_{\pi}})^{\circ}\big\}$, and therefore
		\[
		\left.\left(\mathrm{pr}^{!}(i^{!}L_{\pi})\right)^{\eta}\right|_{X}=\big\{\pi^{\sharp}(\alpha)+\xi+(\mathrm{d}\mathrm{pr})^{*}((\mathrm{d}i)^{*}\alpha)+\iota_{\pi^{\sharp}(\alpha)+\xi}\eta:\ \alpha\in(TX^{\perp_{\pi}})^{\circ}, \xi\in(TX^{\perp_{\pi}})^{*} \big\}.
		\]
		Assume that $\pi^{\sharp}(\alpha)+\xi\in T(TX^{\perp_{\pi}})^{*}|_{X}\cap\left(\mathrm{pr}^{!}(i^{!}L_{\pi})\right)^{\eta}|_{X}$ for $\alpha\in(TX^{\perp_{\pi}})^{\circ}, \xi\in(TX^{\perp_{\pi}})^{*}$. Then $(\mathrm{d}\mathrm{pr})^{*}((\mathrm{d}i)^{*}\alpha)+\iota_{\pi^{\sharp}(\alpha)+\xi}\eta=0$, which implies the following:
		\begin{itemize}
			\item For all $v\in TX$, we get
			\[
			\alpha(v)+\eta\big(\pi^{\sharp}(\alpha)+\xi,v\big)=0\Rightarrow \alpha(v)+\langle j(\xi),v\rangle=0.
			\]
			So $\alpha+j(\xi)\in TX^{\circ}$, and therefore $\pi^{\sharp}(\alpha+j(\xi))\in TX^{\perp_{\pi}}$.
			\item For all $\beta\in(TX^{\perp_{\pi}})^{*}$, we get
			\begin{align*}
				\eta\big(\pi^{\sharp}(\alpha)+\xi,\beta\big)=0&\Rightarrow \pi\big(j(\xi),j(\beta)\big)+\langle \pi^{\sharp}(\alpha),j(\beta)\rangle=0\\
				&\Rightarrow \big\langle \pi^{\sharp}(\alpha+j(\xi)),j(\beta)\big\rangle=0.
			\end{align*}
			Since $j\big((TX^{\perp})^{*}\big)=W^{\circ}$, this shows that $\pi^{\sharp}(\alpha+j(\xi))$ lies in $W$. 
		\end{itemize}
		We now proved that $\pi^{\sharp}(\alpha+j(\xi))\in TX^{\perp_{\pi}}\cap W=\{0\}$. So $\pi^{\sharp}\big(j(\xi)\big)=-\pi^{\sharp}(\alpha)\in TX$, which implies that $j(\xi)\in(TX^{\perp_{\pi}})^{\circ}$, again using exactness of the sequence \eqref{sequence}. But then $j(\xi)\in W^{\circ}\cap(TX^{\perp_{\pi}})^{\circ}=\{0\}$, so that $\xi=0$, which in turn implies that also $\pi^{\sharp}(\alpha)=0$.
	\end{proof}
	
	We denote the Poisson manifold from Proposition \ref{model} by $\big(U,\pi(W,\eta)\big)$, and we refer to it as the \textbf{local model corresponding with} $\mathbf{W}$ \textbf{and} $\boldsymbol{\eta}$. A priori, the construction depends on a choice of complement $W$ and a choice of closed extension $\eta$. We now show that different choices produce isomorphic local models.

	\begin{prop}\label{independent}
		Any two local models for the local Poisson saturation of a coregular submanifold $X\subset (M,\pi)$ are isomorphic around $X$, through a diffeomorphism that restricts to the identity along $X$.
	\end{prop}
	\begin{proof}
		Let $\big(U,\pi(W_{0},\eta_{0})\big)$ and $\big(V,\pi(W_{1},\eta_{1})\big)$ be two local models for the local Poisson saturation of $X$. The idea of the proof is to construct a diffeomorphism between them in two stages, where each stage relies on a Moser argument. We first map the local model $\big(U,\pi(W_{0},\eta_{0})\big)$ to an intermediate local model  $\big(V',\pi(W_{1},\eta'_{1})\big)$, which is defined in terms of the complement $W_{1}$. Then we pull $\big(V',\pi(W_{1},\eta'_{1})\big)$ to the second local model $\big(V,\pi(W_{1},\eta_{1})\big)$. Throughout, we shrink the neighborhoods on which the models are defined, when necessary.

		We interpolate smoothly between the complements $W_{0},W_{1}$ to $TX^{\perp_{\pi}}$ in $TM|_{X}$, as follows.
		Decomposing $W_{1}$ in the direct sum $TM|_{X}=TX^{\perp_{\pi}}\oplus W_{0}$, we find $A\in\Gamma\big(\text{Hom}(W_{0},TX^{\perp_{\pi}})\big)$ such that $W_{1}=\text{Graph}(A)$. If we define $W_{t}:=\text{Graph}(tA)$ for $t\in[0,1]$, then the family $\{W_{t}\}_{t\in[0,1]}$ consists of complements to $TX^{\perp_{\pi}}$, i.e. $TM|_{X}=TX^{\perp_{\pi}}\oplus W_{t}$, and it interpolates between $W_{0}$ and $W_{1}$. Denote by $q_{t}:TM|_{X}\twoheadrightarrow TX^{\perp_{\pi}}$ and $j_{t}:\big(TX^{\perp_{\pi}}\big)^{*}\hookrightarrow T^{*}M|_{X}$ the projection  and inclusion, respectively, induced by the complement $W_{t}$. We first determine the bilinear forms $\sigma_{t}\in\Gamma(\wedge^{2}TX^{\perp_{\pi}})$ and $\tau_{t}\in\Gamma(T^{*}X\otimes TX^{\perp_{\pi}})$, which are defined by the formulas \eqref{st} using the inclusion $j_{t}$, in terms of $\sigma_{0}$ and $\tau_{0}$.
		
		\vspace{0.2cm}
		\noindent
		\underline{Step 1:} We compute $\sigma_{t}$ and $\tau_{t}$.
		
		\vspace{0.1cm}
		\noindent
		For $e+w\in TX^{\perp_{\pi}}\oplus W_{0}=TM|_{X}$, we have
		\begin{align*}
			q_{t}(e+w)&=q_{t}\big(e-tA(w)+w+tA(w)\big)\\
			&=e-tA(w)\\
			&=q_{0}(e+w)-tA(\mathrm{Id}-q_{0})(e+w).
		\end{align*}
		This shows that $q_{t}=q_{0}-tA(\mathrm{Id}-q_{0})$ and therefore $j_{t}=j_{0}-t(\mathrm{Id}-j_{0})A^{*}$. We now compute for $v_{1},v_{2}\in T_{x}X$ and $\xi_{1},\xi_{2}\in (T_{x}X^{\perp_{\pi}})^{*}$: 
		\begin{align}\label{tau}
			\tau_{t}\big((v_{1},\xi_{1}),(v_{2},\xi_{2})\big)&=\langle v_{1},j_{t}(\xi_{2})\rangle - \langle v_{2},j_{t}(\xi_{1})\rangle\nonumber\\
			&=\langle v_{1},j_{0}(\xi_{2})\rangle -t\langle v_{1},(\mathrm{Id}-j_{0})A^{*}\xi_{2}\rangle\nonumber\\
			&\hspace{1cm} - \langle v_{2},j_{0}(\xi_{1})\rangle +t\langle v_{2},(\mathrm{Id}-j_{0})A^{*}\xi_{1}\rangle\nonumber\\
			&=\tau_{0}\big((v_{1},\xi_{1}),(v_{2},\xi_{2})\big)+t\langle A(\mathrm{Id}-q_{0})v_{2},\xi_{1}\rangle-t\langle A(\mathrm{Id}-q_{0})v_{1},\xi_{2}\rangle.
		\end{align}
		Similarly, we obtain
		\begin{align}\label{sigma}
			\sigma_{t}(\xi_{1},\xi_{2})&=\langle q_{t}(\pi^{\sharp}(j_{t}(\xi_{1}))),\xi_{2}\rangle\nonumber\\
			&=\left\langle \left[\big(q_{0}-tA(\mathrm{Id}-q_{0})\big)\pi^{\sharp}\big(j_{0}-t(\mathrm{Id}-j_{0})A^{*}\big)\right](\xi_{1}),\xi_{2}\right\rangle\nonumber\\
			&=\sigma_{0}(\xi_{1},\xi_{2})-t\left\langle \big(q_{0}\pi^{\sharp}(\mathrm{Id}-j_{0})A^{*}\big)(\xi_{1}),\xi_{2}\right\rangle\nonumber\\
			&\hspace{0.5cm}-t\left\langle \big(A(\mathrm{Id}-q_{0})\pi^{\sharp}j_{0}\big)(\xi_{1}),\xi_{2}\right\rangle+t^{2}\left\langle \big(A(\mathrm{Id}-q_{0})\pi^{\sharp}(\mathrm{Id}-j_{0})A^{*}\big)(\xi_{1}),\xi_{2} \right\rangle.
		\end{align}
		
		\vspace{0.2cm}
		\noindent
		\underline{Step 2:} Get closed extensions, smoothly varying in $t\in[0,1]$, of  \[-\sigma_{t}\oplus -\tau_{t}\oplus 0\in\Gamma(\wedge^{2}TX^{\perp_{\pi}})\oplus\Gamma(T^{*}X\otimes TX^{\perp_{\pi}})\oplus\Gamma(\wedge^{2}T^{*}X).\]
		
		\vspace{0.1cm}
		\noindent
		Thanks to \cite[Extension Theorem]{weinstein} and \cite[Relative Poincar\'e Lemma]{weinstein}, we find a one-form $\beta_{1}$, defined on a neighborhood of $X\subset(TX^{\perp_{\pi}})^{*}$, such that
		\[
		\begin{cases}
			\beta_{1}|_{X}=0,\\
			\mathrm{d}\beta_{1}|_{X}\in \Gamma(T^{*}X\otimes TX^{\perp_{\pi}}),\\
			\mathrm{d}\beta_{1}|_{X}\left((v_{1},\xi_{1}),(v_{2},\xi_{2})\right)=\langle A(\mathrm{Id}-q_{0})v_{1},\xi_{2}\rangle-\langle A(\mathrm{Id}-q_{0})v_{2},\xi_{1}\rangle,
		\end{cases}
		\]
		for $(v_{1},\xi_{1}),(v_{2},\xi_{2})\in T_{x}X\oplus \big(T_{x}X^{\perp_{\pi}}\big)^{*}$.
		Similarly, we find one-forms $\beta_{2},\beta_{3}$ defined around $X\subset(TX^{\perp_{\pi}})^{*}$ satisfying
		\[
		\begin{cases}
			\beta_{2}|_{X}=0,\\
			\mathrm{d}\beta_{2}|_{X}\in\Gamma(\wedge^{2}TX^{\perp_{\pi}}),\\
			\mathrm{d}\beta_{2}|_{X}(\xi_{1},\xi_{2})=\left\langle \big(q_{0}\pi^{\sharp}(\mathrm{Id}-j_{0})A^{*}\big)(\xi_{1}),\xi_{2}\right\rangle+\left\langle \big(A(\mathrm{Id}-q_{0})\pi^{\sharp}j_{0}\big)(\xi_{1}),\xi_{2}\right\rangle,
		\end{cases}
		\]
		and
		\[
		\begin{cases}
			\beta_{3}|_{X}=0,\\
			\mathrm{d}\beta_{3}|_{X}\in\Gamma(\wedge^{2}TX^{\perp_{\pi}}),\\
			\mathrm{d}\beta_{3}|_{X}(\xi_{1},\xi_{2})=\left\langle \big(A(\mathrm{Id}-q_{0})\pi^{\sharp}(\mathrm{Id}-j_{0})A^{*}\big)(\xi_{1}),\xi_{2} \right\rangle,
		\end{cases}
		\]
		for $\xi_{1},\xi_{2}\in\big(T_{x}X^{\perp_{\pi}}\big)^{*}$. Using \eqref{tau} and \eqref{sigma}, we see that
		\begin{equation}\label{extension}
			\big(\eta_{0}+t\mathrm{d}\beta_{1}+t\mathrm{d}\beta_{2}-t^{2}\mathrm{d}\beta_{3}\big)|_{X}=-\sigma_{t}\oplus -\tau_{t}\oplus 0.
		\end{equation}
		
		\vspace{0.2cm}
		\noindent
		\underline{Step 3:} A Moser argument pulls $\big(U,\pi(W_{0},\eta_{0})\big)$ to $\big(V',\pi(W_{1},\eta_{0}+\mathrm{d}\beta_{1}+\mathrm{d}\beta_{2}-\mathrm{d}\beta_{3})\big)$.
		
		\vspace{0.1cm}
		\noindent
		By Proposition \ref{model}, we get a path of Dirac structures 
		\[
		\pi_{t}:=\big(\mathrm{pr}^{!}(i^{!}L_{\pi})\big)^{\eta_{0}+t\mathrm{d}\beta_{1}+t\mathrm{d}\beta_{2}-t^{2}\mathrm{d}\beta_{3}}
		\]
		for $t\in[0,1]$, where $\pi_{t}$ is Poisson on a neighborhood $U_{t}$ of $X$ in $\big(TX^{\perp_{\pi}}\big)^{*}$. Note that the set $\bigcup_{t\in[0,1]}\{t\}\times U_{t}$ is open, since it consists of the points $(t,x)$ for which $(\pi_{t})_{x}$ is Poisson. The Tube Lemma implies that $U':=\bigcap_{t\in[0,1]}U_{t}$ is a neighborhood of $X$ on which $\pi_{t}$ is Poisson for all $t\in[0,1]$. Now, these Poisson structures are related by gauge transformations:
		\[
		\pi_{t}=\pi_{0}^{t\mathrm{d}\beta_{1}+t\mathrm{d}\beta_{2}-t^{2}\mathrm{d}\beta_{3}},
		\] 
		where
		\[
		\frac{\mathrm{d}}{\mathrm{d}t}(t\mathrm{d}\beta_{1}+t\mathrm{d}\beta_{2}-t^{2}\mathrm{d}\beta_{3})=-\mathrm{d}(2t\beta_{3}-\beta_{2}-\beta_{1}).
		\]
		A Poisson version of Moser's theorem (e.g. \cite[Theorem 2.11]{meinrenken}) shows that the flow $\Phi_{t}$ of the time-dependent vector field $\pi_{t}^{\sharp}(2t\beta_{3}-\beta_{2}-\beta_{1})$ satisfies $(\Phi_{t})_{*}\pi_{t}=\pi_{0}$, whenever it is defined. Moreover, since the primitive $2t\beta_{3}-\beta_{2}-\beta_{1}$ vanishes along $X$, the flow $\Phi_{t}$ fixes all points in $X$. Now set $\phi:=\Phi_{1}^{-1}$. Shrinking $U$ if necessary, we can assume that $\phi:U\rightarrow V'$ where $V':=\phi(U)$. We then have 
		\[
		\phi:\big(U,\pi(W_{0},\eta_{0})\big)\overset{\sim}{\rightarrow}\big(V',\pi(W_{1},\eta_{0}+\mathrm{d}\beta_{1}+\mathrm{d}\beta_{2}-\mathrm{d}\beta_{3})\big),\hspace{1cm}\phi|_{X}=\text{Id}.
		\]
		
		\vspace{0.2cm}
		\noindent
		\underline{Step 4:} Another Moser argument pulls $\big(V',\pi(W_{1},\eta_{0}+\mathrm{d}\beta_{1}+\mathrm{d}\beta_{2}-\mathrm{d}\beta_{3})\big)$ to $\big(V,\pi(W_{1},\eta_{1})\big)$.
		
		\vspace{0.1cm}
		\noindent
		Both $\eta_{1}$ and $\eta_{0}+\mathrm{d}\beta_{1}+\mathrm{d}\beta_{2}-\mathrm{d}\beta_{3}$ are closed extensions of
		\[
		-\sigma_{1}\oplus -\tau_{1}\oplus 0\in\Gamma(\wedge^{2}TX^{\perp_{\pi}})\oplus\Gamma(T^{*}X\otimes TX^{\perp_{\pi}})\oplus\Gamma(\wedge^{2}T^{*}X),
		\]
		see equation \eqref{extension}. So their difference $\eta_{1}-(\eta_{0}+\mathrm{d}\beta_{1}+\mathrm{d}\beta_{2}-\mathrm{d}\beta_{3})$ is exact around $X$ with a primitive $\gamma$ that vanishes along $X$, by the Relative Poincar\'e Lemma. Denote
		\[
		\pi_{0}':=\big(\mathrm{pr}^{!}(i^{!}L_{\pi})\big)^{\eta_{0}+\mathrm{d}\beta_{1}+\mathrm{d}\beta_{2}-\mathrm{d}\beta_{3}},\hspace{1cm}\pi_{t}':=\left(\pi_{0}'\right)^{t\mathrm{d}\gamma},
		\]
		for $t\in[0,1]$. Since $\pi_{0}'$ is Poisson on $V'$ and $\mathrm{d}\gamma|_{X}=0$, we see that $\pi_{t}'$ is Poisson on a neighborhood $V'_{t}$ of $X$ in $(TX^{\perp_{\pi}})^{*}$. Using the Tube Lemma as in Step 3, we find a neighborhood $O$ of $X$ in $(TX^{\perp_{\pi}})^{*}$ such that $\pi_{t}'$ is Poisson on $O$ for all $t\in[0,1]$. The Moser Theorem \cite[Theorem 2.11]{meinrenken} implies that the flow $\Psi_{t}$ of the time-dependent vector field $-(\pi_{t}')^{\sharp}(\gamma)$ satisfies $(\Psi_{t})_{*}\pi_{t}'=\pi_{0}'$, whenever it is defined. Moreover, since $\gamma|_{X}=0$, the flow $\Psi_{t}$ fixes all points of $X$. Now set $\psi:=\Psi_{1}^{-1}$. Shrinking both $V'$ and $V$ if necessary, we can assume that $\psi:V'\rightarrow V$. We then have
		\[
		\psi:\big(V',\pi(W_{1},\eta_{0}+\mathrm{d}\beta_{1}+\mathrm{d}\beta_{2}-\mathrm{d}\beta_{3})\big)\overset{\sim}{\rightarrow}\big(V,\pi(W_{1},\eta_{1})\big),\hspace{1cm}\psi|_{X}=\text{Id}.
		\] 
		The diffeomorphism $\psi\circ\phi$ now satisfies the criteria: it fixes points in $X$ and defines a Poisson diffeomorphism
		\[
		\psi\circ\phi:\big(U,\pi(W_{0},\eta_{0})\big)\overset{\sim}{\rightarrow}\big(V,\pi(W_{1},\eta_{1})\big). \qedhere
		\]
	\end{proof}
	
	It is now justified to call $\big(U,\pi(W,\eta)\big)$ \textbf{the local model} for the local Poisson saturation of the coregular submanifold $X\subset (M,\pi)$.

	\section{The normal form}\label{sec:third}
	We now show that the local Poisson saturation of a coregular submanifold $X\subset (M,\pi)$ is isomorphic around $X$ to the local model $\big(U,\pi(W,\eta)\big)$ constructed in Proposition \ref{model}. 
	We will use the theory of dual pairs in Dirac geometry, as developed in \cite{dirac}. We first need a lemma, which describes how to obtain a weak Dirac dual pair out of the self-dual pair \eqref{pair}
	\[
	\begin{tikzcd}
		(M,\pi)&(\Sigma,\Omega_{\chi})\arrow{r}{\exp_{\chi}}\arrow[l,"\mathrm{pr}",swap]&(M,-\pi)
	\end{tikzcd}
	\]
	whenever a coregular submanifold $X\subset (M,\pi)$ is given. Recall from the proof of Theorem \ref{satsmooth} that the local Poisson saturation $(P,\pi_{P})$ of $X\subset (M,\pi)$ is given by $\exp_{\chi}(\Sigma|_{X})$.
	
	\begin{lemma}\label{diracpair}
		Let $i:X\hookrightarrow(M,\pi)$ be a coregular submanifold with local Poisson saturation $(P,\pi_{P})$. Then the following is a weak Dirac dual pair, in the sense of \cite{dirac}:
		\[
		\begin{tikzcd}
			(X,i^{!}L_{\pi})&\big(\Sigma|_{X},\text{Gr}(\Omega_{\chi}|_{X})\big)\arrow{r}{\exp_{\chi}}\arrow[l,"\mathrm{pr}",swap]&(P,-L_{\pi_{P}}).
		\end{tikzcd}
		\]  
		This means that $\Omega_{\chi}|_{X}$ is a closed two-form on $\Sigma|_{X}$, that $\mathrm{pr}$ and $\exp_{\chi}$ are surjective forward Dirac submersions, and that
		\begin{align}
			&\big(\Omega_{\chi}|_{X}\big)(S_{1},S_{2})=0,\label{property1}\\
			&rk(S_{1}\cap K\cap S_{2})=\dim \Sigma|_{X}-\dim X-\dim P,\label{property2}
		\end{align}
		where $S_{1}:=\ker \mathrm{d}\mathrm{pr},S_{2}:=\ker \mathrm{d}\exp_{\chi}$ and $K:=\ker\big(\Omega_{\chi}|_{X}\big)$.
	\end{lemma}
	\begin{proof}
		It is clear that $\mathrm{pr}$ is a surjective submersion. The fact that $\exp_{\chi}$ is a surjective submersion follows from the proof of Theorem \ref{satsmooth}. Because \eqref{pair} is a dual pair, property \eqref{property1} is automatic. To see that $\mathrm{pr}:\big(\Sigma|_{X},\text{Gr}(\Omega_{\chi}|_{X})\big)\rightarrow(X,i^{!}L_{\pi})$ is forward Dirac, consider the following commutative diagram of Dirac manifolds and smooth maps:
		\[
		\begin{tikzcd}[column sep=large, row sep=large]
			\big(\Sigma|_{X},\text{Gr}(\Omega_{\chi}|_{X})\big)\arrow[r,"\mathrm{pr}"]\arrow[d,hookrightarrow,"i'"] & (X,i^{!}L_{\pi})\arrow[d,hookrightarrow,"i"]\\
			(\Sigma,\text{Gr}(\Omega_{\chi}))\arrow[r,"\mathrm{pr}"] & (M,L_{\pi})
		\end{tikzcd}.
		\]
		The maps $i'$ on the left and $i$ on the right are backward Dirac by definition, and the bottom map $\mathrm{pr}$ is forward Dirac because of the dual pair \eqref{pair}. Since the bottom map $\mathrm{pr}$ is a submersion, we can apply \cite[Lemma 3]{maps} to obtain that also the map at the top $\mathrm{pr}:\big(\Sigma|_{X},\text{Gr}(\Omega_{\chi}|_{X})\big)\rightarrow(X,i^{!}L_{\pi})$ is forward Dirac. 
		
		Similarly, we get that $\exp_{\chi}:\big(\Sigma|_{X},\text{Gr}(\Omega_{\chi}|_{X})\big)\rightarrow(P,-L_{\pi_{P}})$ is forward Dirac considering the diagram
		\[
		\begin{tikzcd}[column sep=large, row sep=large]
			\big(\Sigma|_{X},\text{Gr}(\Omega_{\chi}|_{X})\big)\arrow[r,"\exp_{\chi}"]\arrow[d,hookrightarrow,"i'"] & (P,-L_{\pi_{P}})\arrow[d,hookrightarrow,"i"]\\
			(\Sigma,\text{Gr}(\Omega_{\chi}))\arrow[r,"\exp_{\chi}"] & (M,-L_{\pi})
		\end{tikzcd}.
		\]
		Here the map $i'$ is backward Dirac, the map $i$ is backward (and forward) Dirac, and the bottom map $\exp_{\chi}$ is forward Dirac because of the dual pair \eqref{pair}. Again, the map $\exp_{\chi}$ on the bottom is a submersion, so we can apply \cite[Lemma 3]{maps} to obtain that also the map $\exp_{\chi}:\big(\Sigma|_{X},\text{Gr}(\Omega_{\chi}|_{X})\big)\rightarrow(P,-L_{\pi_{P}})$ at the top is forward Dirac.
		
		It remains to check that property \eqref{property2} holds. For $(x,\xi)\in \Sigma|_{X}$, we have
		\begin{equation*}
			K_{(x,\xi)}
			=\pi_{\chi}^{\sharp}\left((\mathrm{d}\mathrm{pr})^{*}_{(x,\xi)}T_{x}X^{\circ}\right)\cap T_{(x,\xi)}(T^{*}M|_{X}),
		\end{equation*}
		where $\pi_{\chi}:=\Omega_{\chi}^{-1}$. Consequently, we obtain
		\begin{equation*}
			(S_{1})_{(x,\xi)}\cap K_{(x,\xi)}
			=\pi_{\chi}^{\sharp}\left((\mathrm{d}\mathrm{pr})^{*}_{(x,\xi)}\left(T_{x}X^{\circ}\cap\ker\pi_{x}^{\sharp}\right)\right),
		\end{equation*}
		using that the left leg of the dual pair \eqref{pair} is a Poisson map.
		The equality \eqref{T} in the proof of Lemma \ref{realization} shows that
		\[
		(S_{2})_{(x,\xi)}=\pi_{\chi}^{\sharp}\left((\mathrm{d}\mathrm{pr})_{(x,\xi)}^{*}\big(T_{x}X^{\perp_{\pi}}\big)^{\circ}\right),
		\]
		so we obtain
		\[
		(S_{1})_{(x,\xi)}\cap K_{(x,\xi)}\cap(S_{2})_{(x,\xi)}=\pi_{\chi}^{\sharp}\left((\mathrm{d}\mathrm{pr})^{*}_{(x,\xi)}\left(T_{x}X^{\circ}\cap\ker\pi_{x}^{\sharp}\right)\right).
		\]
		Consequently,
		\begin{align*}
			rk(S_{1}\cap K\cap S_{2})_{(x,\xi)}&=\dim\big(T_{x}X^{\circ}\cap\ker\pi_{x}^{\sharp}\big)\\
			&=\dim(T_{x}X^{\circ})-\dim(T_{x}X^{\perp_{\pi}})\\
			&=(\dim M-\dim X)-(\dim P-\dim X)\\
			&=\dim \Sigma|_{X}-\dim X-\dim P.
		\end{align*}
		So also property \eqref{property2} holds, and this finishes the proof.
	\end{proof}

	We are now ready to state the main results of this section.
	
	\begin{thm}\label{normalform}
		Let $X\subset (M,\pi)$ be a coregular submanifold with local Poisson saturation $(P,\pi_{P})$. Choose a complement $W$ to $TX^{\perp_{\pi}}$ in $TM|_{X}$ and denote by $j:(TX^{\perp_{\pi}})^{*}\hookrightarrow T^{*}M|_{X}$ the corresponding inclusion. Then $-j^{*}(\Omega_{\chi}|_{X})\in\mathcal{E}_{W}(-\sigma-\tau)$. Moreover, the corresponding local model $\left(U,\pi\big(W,-j^{*}(\Omega_{\chi}|_{X})\big)\right)$ is isomorphic with $(P,\pi_{P})$ around $X$. Explicitly, a Poisson diffeomorphism onto a neighborhood of $X$ is given by
		\[
		\exp_{\chi}\circ j:\left(U,\pi\big(W,-j^{*}(\Omega_{\chi}|_{X})\big)\right)\overset{\sim}{\rightarrow}(P,\pi_{P}).
		\]
	\end{thm}
	
	We will denote by $\mathrm{pr}_{M}$ and $\mathrm{pr}_{X}$ the bundle projections $T^{*}M|_{X}\rightarrow X$ and $\big(TX^{\perp_{\pi}}\big)^{*}\rightarrow X$, respectively. So $\mathrm{pr}_{M}\circ j = \mathrm{pr}_{X}$.
	
	\begin{proof}
		We first check that $-j^{*}(\Omega_{\chi}|_{X})\in\mathcal{E}_{W}(-\sigma-\tau)$. The fact that $-j^{*}(\Omega_{\chi}|_{X})$ restricts along $X\subset(TX^{\perp_{\pi}})^{*}$ as required in \eqref{restr} is an immediate consequence of the following equality \cite[Lemma 24]{transversals}: 
		\[
		\Omega_{\chi}\big((v_{1},\xi_{1}),(v_{2},\xi_{2})\big)=\langle v_{1},\xi_{2}\rangle - \langle v_{2},\xi_{1}\rangle+\pi(\xi_{1},\xi_{2}),
		\]
		where $(v_{1},\xi_{1}),(v_{2},\xi_{2})\in T_{x}(T^{*}M)=T_{x}M\oplus T_{x}^{*}M$ for  $x\in M$. 
		
		To prove the second statement, we apply \cite[Proposition 6]{dirac} to the weak dual pair constructed in Lemma \ref{diracpair},
		\[
		\begin{tikzcd}
			(X,i^{!}L_{\pi})&\big(\Sigma|_{X},\text{Gr}(\Omega_{\chi}|_{X})\big)\arrow{r}{\exp_{\chi}}\arrow[l,"\mathrm{pr}_{M}",swap]&(P,-L_{\pi_{P}}),
		\end{tikzcd}
		\]
		and we get the following equality of Dirac structures on $\Sigma|_{X}$:
		\begin{equation}\label{eq:dirbeforepull}
			(\mathrm{pr}_{M}^{!}(i^{!}L_{\pi}))^{-\Omega_{\chi}|_{X}}=\exp_{\chi}^{!}L_{\pi_{P}}.
		\end{equation}
		Since the map $j:(TX^{\perp_{\pi}})^{*}\hookrightarrow T^{*}M|_{X}$ is transverse to this Dirac structure, we can pull back the equality \eqref{eq:dirbeforepull} to  $j^{-1}(\Sigma|_{X})\cong j(TX^{\perp_{\pi}})^{*}\cap\Sigma$,
			which yields
			\begin{equation}\label{comp}
				(\mathrm{pr}_{X}^{!}(i^{!}L_{\pi}))^{-j^{*}(\Omega_{\chi}|_{X})}=(\exp_{\chi}\circ j)^{!}L_{\pi_{P}}.
		\end{equation}
		The left hand side of \eqref{comp} is Poisson on a neighborhood $U\subset j^{-1}(\Sigma|_{X})\subset(TX^{\perp_{\pi}})^{*}$, where it defines the local model $\left(U,\pi\big(W,-j^{*}(\Omega_{\chi}|_{X})\big)\right)$. Moreover, by the proof of Theorem \ref{satsmooth}, we know that $\exp_{\chi}\circ j$ takes $j^{-1}(\Sigma|_{X})$ diffeomorphically onto $P$. So we obtain that
		\[
		\exp_{\chi}\circ j:\left(U,\pi\big(W,-j^{*}(\Omega_{\chi}|_{X})\big)\right)\rightarrow(P,\pi_{P})
		\]
		is a Poisson diffeomorphism onto a neighborhood of $X\subset (P,\pi_{P})$, as desired.
	\end{proof}
	
	We now combine Proposition \ref{independent} and Theorem \ref{normalform}. Also noticing that any local model $(U,\pi(W,\eta))$ for the local Poisson saturation of $X$ is constructed out of the restriction $\pi|_{X}$, we obtain the following.
	
	\begin{cor}\label{cor:normalform}
		Let $X\subset(M,\pi)$ be a coregular submanifold with local Poisson saturation $(P,\pi_P)$. For any choice of complement $W$ to $TX^{\perp_{\pi}}$ and closed extension $\eta\in\mathcal{E}_{W}(-\sigma-\tau)$, the corresponding local model $(U,\pi(W,\eta))$ is Poisson diffeomorphic around $X$ with $(P,\pi_P)$. In particular, up to Poisson diffeomorphism, the local Poisson saturation is determined by the restriction of $\pi$ along $X$.
	\end{cor}
	
	In general, one needs the full information of $\pi|_{X}$ in order to determine the local Poisson saturation of $X$. We will see in the next section that for certain coregular submanifolds $X$, only part of this information is required.

	\begin{remark}
		We outline an alternative proof for our normal form result, relying on the fact that a coregular submanifold $X\subset(M,\pi)$ is a transversal in its local Poisson saturation. This allows one to use the normal form around Dirac transversals \cite[Thm. 5.1]{eulerlike}, \cite[\S 7]{dirac} instead of Theorem \ref{normalform}, which in combination with Proposition \ref{independent} yields Corollary \ref{cor:normalform}. We now elaborate on this, using explicitly the normal form in \cite[\S 7]{dirac} because it is the closest in spirit to the arguments in this note.
		
		A choice of complement $TM|_{X}=TX^{\perp_{\pi}}\oplus W$ gives an inclusion $j:(TX^{\perp_{\pi}})^{*} \hookrightarrow T^{*}M|_{X}$, and an identification of the normal bundle
		\[
		\pi^{\sharp}\circ j: (TX^{\perp_{\pi}})^{*}\rightarrow \pi^{\sharp}(j(TX^{\perp_{\pi}})^{*})\cong TP|_{X}/TX.
		\]
		According to \cite[\S 7]{dirac}, the Dirac manifold $(P,L_{\pi_P})$ is isomorphic around $X$ with
		\begin{equation}\label{eq:normdirac}
			\left(U\subset(TX^{\perp_{\pi}})^{*}, \big(\mathrm{pr}^{!}(i^{!}(L_{\pi_P}))\big)^{-\rho^{*}\omega|_{X}}\right).
		\end{equation}
		Here $\rho$ is a splitting of the exact sequence
		\begin{equation}\label{eq:seqdirac}
			0\longrightarrow i^{!}L_{\pi_P}\longrightarrow L_{\pi_P}|_{X}\longrightarrow (TX^{\perp_{\pi}})^{*}\longrightarrow 0,
		\end{equation}
		where the last arrow is the anchor map $\mathrm{pr}_{T}:L_{\pi_P}|_{X}\rightarrow TP|_{X}$ composed with the projection to the normal bundle $TP|_{X}/TX\cong(TX^{\perp_{\pi}})^{*}$. The two-form $\omega$ appearing in \eqref{eq:normdirac} is defined choosing a spray on $L_{\pi_P}$, see \eqref{eq:sprayform} for the precise formula. Hence, to prove our normal form for one specific choice of local model, as we did in Theorem \ref{normalform}, we just have to show that there is a splitting $\rho$ of the sequence \eqref{eq:seqdirac} satisfying
		\[
		\rho^{*}\omega|_{X}=\sigma+\tau,
		\]
		where $\sigma$ and $\tau$ were defined in \eqref{st}. We claim that such a splitting is given by
		\[
		\rho:(TX^{\perp_{\pi}})^{*}\rightarrow L_{\pi_P}|_{X}:\xi\mapsto \pi^{\sharp}(j(\xi))+(\mathrm{d}\iota)^{*}(j(\xi)),
		\]
		where $\iota:P\hookrightarrow M$ is the inclusion. Indeed, choosing $(v_1,\xi_1),(v_2,\xi_2)\in T_{x}(TX^{\perp_{\pi}})^{*}$, we get
		\begin{align*}
			&\rho^{*}\omega|_{X}\big((v_1,\xi_1),(v_2,\xi_2)\big)\\
			&\hspace{0.5cm}=\omega|_{X}\left(\big(v_1,\pi^{\sharp}(j(\xi_1))+(\mathrm{d}\iota)^{*}(j(\xi_1))\big),\big(v_2,\pi^{\sharp}(j(\xi_2))+(\mathrm{d}\iota)^{*}(j(\xi_2))\big)\right)\\
			&\hspace{0.5cm}=(\mathrm{d}\iota)^{*}(j(\xi_2))\left(v_1+\frac{1}{2}\pi^{\sharp}(j(\xi_1))\right)-(\mathrm{d}\iota)^{*}(j(\xi_1))\left(v_2+\frac{1}{2}\pi^{\sharp}(j(\xi_2))\right)\\
			&\hspace{0.5cm}=\langle v_1,j(\xi_2)\rangle-\langle v_2,j(\xi_1)\rangle+\pi\big(j(\xi_1),j(\xi_2)\big)\\
			&\hspace{0.5cm}=(\sigma+\tau)\big((v_1,\xi_1),(v_2,\xi_2)\big),
		\end{align*}
		using \cite[eq.~6]{dirac} in the second equality. This proves that the local Poisson saturation  $(P,\pi_P)$ is isomorphic around $X$ with the local model $\left(U,\pi\big(W,-\rho^{*}\omega|_{X}\big)\right)$. Along with Proposition \ref{independent}, this gives an alternative proof for our normal form in Corollary \ref{cor:normalform}.
	\end{remark}

	\section{Some particular cases}\label{sec:fourth}
	
	We proved that the local model $\big(U,\pi(W,\eta)\big)$ described in Proposition \ref{model} depends neither on the choice of complement $W$ to $TX^{\perp_{\pi}}$ in $TM|_{X}$, nor on the choice of closed extension $\eta$. We now show that, for certain classes of coregular submanifolds $X\subset (M,\pi)$, a good choice of complement and/or closed extension simplifies the normal form considerably. Some of our results recover well-known normal form and rigidity statements around distinguished submanifolds in symplectic and Poisson geometry.
	
	\subsection{Submanifolds in symplectic geometry}
	\leavevmode
	\vspace{0.1cm}
	
	Recall that, if $(M,\omega)$ is a symplectic manifold and $N\subset M$ is any submanifold, then the restriction of $\omega$ to $TM|_{N}$ determines the symplectic form $\omega$ on a neighborhood of $N$ (see \cite[Theorem 4.1]{lagrangians}). We can recover this result from our normal form, as follows.
	
	First note that, in case $\pi=\omega^{-1}$ is symplectic, any submanifold $X\subset(M,\pi)$ is coregular since $TX^{\perp_{\pi}}=TX^{\perp_{\omega}}$, where $TX^{\perp_{\omega}}=\{v\in TM|_{X}:\omega(v,w)=0\ \forall w\in TX\}$ denotes the symplectic orthogonal of $X$. Next, the local Poisson saturation $(P,\pi_{P})$ of $X$ is an embedded submanifold of $M$ of dimension $\dim X + rk\big(\pi^{\sharp}(TX^{\perp_{\pi}})^{*}\big)$, by the equality \eqref{normal}. 
	So if $\pi$ is symplectic, then $P\subset M$ is a neighborhood of $X$. Finally, the Poisson structure $\pi(W,\eta)=\big(\mathrm{pr}^{!}(i^{!}L_{\pi})\big)^{\eta}$ from the local model is determined by the restriction $\pi|_{X}$.
	
	In conclusion, our normal form shows that, for any submanifold $X$ of the symplectic manifold $(M,\pi)$, the restriction $\pi|_{X}$ determines $\pi$ on a neighborhood of $X\subset M$, which recovers the aforementioned rigidity result in symplectic geometry.

	\subsection{Poisson transversals}
	\leavevmode
	\vspace{0.1cm}
	
	A submanifold $X$ of a Poisson manifold $(M,\pi)$ is called a Poisson transversal if it meets each symplectic leaf transversally and symplectically, that is
	\[
	TX\oplus TX^{\perp_{\pi}}=TM|_{X}.
	\]
	In the local model of Proposition \eqref{model}, we can take $TX$ as a canonical complement to $TX^{\perp_{\pi}}$ in $TM|_{X}$. Then the associated embedding $j:(TX^{\perp_{\pi}})^{*}\hookrightarrow T^{*}M|_{X}$ identifies $(TX^{\perp_{\pi}})^{*}$ with $TX^{\circ}$. The following simplifications occur in the local model:
	\begin{itemize}
		\item The pullback $i^{!}L_{\pi}$ of the Dirac structure $L_{\pi}$ to $X$ defines a Poisson structure on $X$ \cite[Lemma 3]{transversals}, which we denote by $\pi_{X}\in\Gamma(\wedge^{2}TX)$. 
		\item Consider $\sigma\in\Gamma(\wedge^{2}TX^{\perp_{\pi}})$ and $\tau\in\Gamma(T^{*}X\otimes TX^{\perp_{\pi}})$ defined in \eqref{st}:
		\begin{align*}
			&\sigma(\xi_{1},\xi_{2})=\pi\big(j(\xi_{1}),j(\xi_{2})\big),\nonumber\\
			&\tau\big((v_{1},\xi_{1}),(v_{2},\xi_{2})\big)=\langle v_{1},j(\xi_{2})\rangle - \langle v_{2},j(\xi_{1})\rangle,
		\end{align*}
		for $\xi_{1},\xi_{2}\in(T_{x}X^{\perp_{\pi}})^{*}$ and $v_{1},v_{2}\in T_{x}X$. Since $j\left((TX^{\perp_{\pi}})^{*}\right)=TX^{\circ}$, we get that $\tau\equiv 0$, and  
		since the restriction of $\pi$ to the conormal bundle $TX^{\circ}$ is fiberwise non-degenerate, we get a symplectic vector bundle $\left((TX^{\perp_{\pi}})^{*},\sigma\right)$.
	\end{itemize}
	Moreover, since $X$ is a transversal, its local Poisson saturation $(P,\pi_{P})$ is in fact a neighborhood of $X$ in $M$. In conclusion, our normal form shows that a neighborhood of $X$ in $(M,\pi)$ is Poisson diffeomorphic with a neighborhood of $X$ in $(TX^{\perp_{\pi}})^{*}$, endowed with the Poisson structure
	\[
	\big(\mathrm{pr}^{!}(L_{\pi_{X}})\big)^{\eta},
	\]
	where $\eta$ is a closed extension of $-\sigma$. This is exactly the normal form established in \cite{transversals}.
	
	\subsection{Coregular coisotropic submanifolds}
	\leavevmode
	\vspace{0.1cm}
	
	Recall that a submanifold $N$ of a symplectic manifold $(M,\omega)$ is called coisotropic if its symplectic orthogonal $TN^{\perp_{\omega}}$ is contained in $TN$. Gotay's theorem \cite{gotay} provides a normal form for $\omega$ around $N$, which is obtained as follows. Choose a complement to $TN^{\perp_{\omega}}$ inside $TN$, and denote by $j:(TN^{\perp_{\omega}})^{*}\hookrightarrow T^{*}N$ the induced inclusion. On the total space of the vector bundle $\mathrm{pr}:(TN^{\perp_{\omega}})^{*}\rightarrow N$, one gets a closed two-form 
	\[
	\mathrm{pr}^{*}(i^{*}\omega)+j^{*}\omega_{can},
	\] 
	where $i^{*}\omega$ is the pullback of $\omega$ to $N$ and $\omega_{can}$ is the canonical symplectic form on $T^{*}N$. This two-form is non-degenerate on a neighborhood of the zero section $N\subset(TN^{\perp_{\omega}})^{*}$, and $(M,\omega)$ is isomorphic with $\big((TN^{\perp_{\omega}})^{*},\mathrm{pr}^{*}(i^{*}\omega)+j^{*}\omega_{can}\big)$ around $N$. In particular, the pullback $i^{*}\omega\in\Gamma(\wedge^{2}T^{*}N)$ determines $\omega$ on a neighborhood of $N\subset M$.

	More generally, recall that a submanifold $X$ of a Poisson manifold $(M,\pi)$ is  coisotropic if 
	$TX^{\perp_{\pi}}\subset TX$. In this subsection, we prove a Poisson version of Gotay's theorem by specializing our normal form to coregular submanifolds $i:X\hookrightarrow(M,\pi)$ that are coisotropic. 
	Mimicking Gotay's construction, we choose a complement $TX=TX^{\perp_{\pi}}\oplus G$ to get an inclusion $j:(TX^{\perp_{\pi}})^{*}\hookrightarrow T^{*}X$, and we obtain a Dirac structure
		\begin{equation}\label{coisomodel}
			\big(\mathrm{pr}^{!}(i^{!}L_{\pi})\big)^{j^{*}\omega_{can}}
		\end{equation}
		on $(TX^{\perp_{\pi}})^{*}$, where $\omega_{can}$ denotes the canonical symplectic form on $T^{*}X$.

	\begin{cor}[Poisson version of Gotay's Theorem]\label{coiso}
		Let $X\subset(M,\pi)$ be coregular coisotropic. The Dirac structure \eqref{coisomodel} defines a Poisson structure on a neighborhood  $U\subset(TX^{\perp_{\pi}})^{*}$ of $X$, which is Poisson diffeomorphic around $X$ with the local Poisson saturation of $X$.
	\end{cor}
	\begin{proof}
		It suffices to show that the Dirac structure \eqref{coisomodel} is diffeomorphic around $X$ with a local model for the local Poisson saturation of $X$. By Lemma \ref{compprep} in the next subsection, the splitting $TX=TX^{\perp_{\pi}}\oplus G$ induces a splitting $TM|_{X}=TX^{\perp_{\pi}}\oplus W_{G}$, where
		\[
		\pi^{\sharp}((W_{G})^{\circ})\subset W_{G}\hspace{0.5cm}\text{and}\hspace{0.5cm}W_{G}\cap TX=G.
		\] 
		Denote by $\tilde{j}:(TX^{\perp_{\pi}})^{*}\hookrightarrow T^{*}M|_{X}$ the inclusion induced by the complement $W_{G}$; it embeds $(TX^{\perp_{\pi}})^{*}$ into $T^{*}M|_{X}$ as $(W_{G})^{\circ}$. Consider $\sigma\in\Gamma(\wedge^{2}TX^{\perp_{\pi}})$ and $\tau\in\Gamma(T^{*}X\otimes TX^{\perp_{\pi}})$ as defined in \eqref{st}:
		\begin{align*}
			&\sigma(\xi_{1},\xi_{2})=\pi\big(\tilde{j}(\xi_{1}),\tilde{j}(\xi_{2})\big),\nonumber\\
			&\tau\big((v_{1},\xi_{1}),(v_{2},\xi_{2})\big)=\langle v_{1},\tilde{j}(\xi_{2})\rangle - \langle v_{2},\tilde{j}(\xi_{1})\rangle,
		\end{align*}
		for $\xi_{1},\xi_{2}\in\big(T_{x}X^{\perp_{\pi}}\big)^{*}$ and $v_{1},v_{2}\in T_{x}X$. Since $\pi^{\sharp}((W_{G})^{\circ})\subset W_{G}$, we have $\sigma\equiv 0$, and since $W_{G}\cap TX=G$, we have 
		\begin{align*}
			\tau\big((v_{1},\xi_{1}),(v_{2},\xi_{2})\big)&=\langle v_{1},\tilde{j}(\xi_{2})\rangle - \langle v_{2},\tilde{j}(\xi_{1})\rangle\\
			&=\langle v_{1},j(\xi_{2})\rangle - \langle v_{2},j(\xi_{1})\rangle\\
			&=\left.(j^{*}\omega_{can})\right|_{X}\big((v_{1},\xi_{1}),(v_{2},\xi_{2})\big)
		\end{align*}
		for $\xi_{1},\xi_{2}\in(T_{x}X^{\perp_{\pi}})^{*}$ and $v_{1},v_{2}\in T_{x}X$. 
		This shows that $\big(U,\pi(W_{G},-j^{*}\omega_{can})\big)$ is a local model for the local Poisson saturation of $X$, where $U\subset(TX^{\perp_{\pi}})^{*}$ is a suitable neighborhood of $X$. Note however that the Dirac structure \eqref{coisomodel} still differs by a sign from this model; we now remedy this. Shrinking $U$ if necessary, we can assume that $U$ is invariant under fiberwise multiplication by $-1$. Denoting this map by $m_{-1}$, we have
		\[
		m_{-1}^{!}\left((\mathrm{pr}^{!}(i^{!}L_{\pi}))^{j^{*}\omega_{can}}\right)=\big((\mathrm{pr}\circ m_{-1})^{!}i^{!}L_{\pi}\big)^{(j\circ m_{-1})^{*}\omega_{can}}=(\mathrm{pr}^{!}(i^{!}L_{\pi}))^{-j^{*}\omega_{can}},
		\]
		the latter being the Poisson structure $\big(U,\pi(W_{G},-j^{*}\omega_{can})\big)$. This shows that the Dirac structure \eqref{coisomodel} is in fact Poisson on $U$, and that it is Poisson diffeomorphic around $X$ with the local Poisson saturation of $X$.
	\end{proof}
	
	In particular, the pullback Dirac structure $i^{!}L_{\pi}$ determines a neighborhood of $X$ in its local Poisson saturation, up to Poisson diffeomorphism. Indeed, if one knows $i^{!}L_{\pi}$, then one also knows $TX^{\perp_{\pi}}=i^{!}L_{\pi}\cap TX$, hence one can construct the local model \eqref{coisomodel} which recovers the local Poisson saturation up to Poisson diffeomorphism around $X$. 
	In the next subsection, we generalize this result to the class of coregular pre-Poisson submanifolds.
	
	\subsection{Coregular pre-Poisson submanifolds}
	\leavevmode
	
	\vspace{0.1cm}
	Recall that, given a symplectic manifold $(M,\omega)$, a submanifold $i:N\hookrightarrow (M,\omega)$ is said to be of constant rank if the pullback $i^{*}\omega$ has constant rank. Marle's constant rank theorem \cite{marle} states that a neighborhood of a constant rank submanifold $i:N\hookrightarrow (M,\omega)$ is determined by the pullback $i^{*}\omega$ and the symplectic vector bundle $\big(TN^{\perp_{\omega}}/\big(TN^{\perp_{\omega}}\cap TN\big),\omega\big)$.
	
	Generalizing this notion to Poisson geometry, a submanifold $X$ of a Poisson manifold $(M,\pi)$ is called pre-Poisson if $TX+TX^{\perp_{\pi}}$ has constant rank \cite{pre-poisson}. It is equivalent to ask that the bundle map $\mathrm{pr}\circ\pi^{\sharp}:TX^{\circ}\rightarrow TX^{\perp_{\pi}}\rightarrow TM|_{X}/TX$ has constant rank. Examples include Poisson transversals (in which case $\mathrm{pr}\circ\pi^{\sharp}$ is an isomorphism) and coisotropic submanifolds (in which case $\mathrm{pr}\circ\pi^{\sharp}$ is the zero map). If $X$ is coregular pre-Poisson, i.e. $TX^{\perp_{\pi}}$ has constant rank, then its characteristic distribution $TX^{\perp_{\pi}}\cap TX$ also has constant rank. 
	
	In this subsection, we prove a Poisson version of Marle's theorem by specializing our normal form to coregular pre-Poisson submanifolds. We need the following auxiliary result.
	
	\begin{lemma}\label{compprep}
		Let $X\subset (M,\pi)$ be a coregular pre-Poisson submanifold. For any choice of splittings $TX=(TX^{\perp_{\pi}}\cap TX)\oplus G$ and $TX^{\perp_{\pi}}=(TX^{\perp_{\pi}}\cap TX)\oplus H$, there exists a complement $TM|_{X}=(TX^{\perp_{\pi}}\cap TX)\oplus H\oplus W_{G,H}$ such that
		\[
		\pi^{\sharp}\left(\big(H+W_{G,H}\big)^{\circ}\right)\subset W_{G,H}\hspace{0.5cm}\text{and}\hspace{0.5cm}W_{G,H}\cap TX=G .
		\]
	\end{lemma}
	\begin{proof}
		We have in particular that
		\begin{equation}\label{prep}
			TX+TX^{\perp_{\pi}}=(TX^{\perp_{\pi}}\cap TX)\oplus G\oplus H.
		\end{equation} 
		The proof is divided into four steps. 
		
		\vspace{0.1cm}
		\noindent
		\underline{Step 1:} $\pi^{\sharp}\big((G+H)^{\circ}\big)$ has constant rank, equal to twice the rank of $TX^{\perp_{\pi}}\cap TX$.
		
		\vspace{0.1cm}
		\noindent
		Since $\ker\pi^{\sharp}\subset (TX^{\perp_{\pi}})^{\circ}\subset (TX^{\perp_{\pi}}\cap TX)^{\circ}$, we have
		\begin{align*}
			\ker\pi^{\sharp}\cap (G+H)^{\circ}&=\ker\pi^{\sharp}\cap (TX^{\perp_{\pi}}\cap TX)^{\circ}\cap (G+H)^{\circ}\\
			&=\ker\pi^{\sharp}\cap\big((TX^{\perp_{\pi}}\cap TX)+G+H\big)^{\circ}\\
			&=\ker\pi^{\sharp}\cap(TX+TX^{\perp_{\pi}})^{\circ}\\
			&=\ker\pi^{\sharp}\cap TX^{\circ}\cap (TX^{\perp_{\pi}})^{\circ}\\
			&=\ker\pi^{\sharp}\cap TX^{\circ}.
		\end{align*}
		Since $X$ is coregular, the latter has constant rank, which shows that also $\pi^{\sharp}\big((G+H)^{\circ}\big)$ has constant rank. Explicitly,
		\begin{align*}
			rk\left(\pi^{\sharp}\big((G+H)^{\circ}\big)\right)&=\dim M - rk(G+H)-rk\big(\ker\pi^{\sharp}\cap (G+H)^{\circ}\big)\\
			&=\dim M-rk(TX+TX^{\perp_{\pi}})+rk(TX^{\perp_{\pi}}\cap TX)-rk\big(\ker\pi^{\sharp}\cap TX^{\circ}\big)\\
			&=\dim M-rk(TX+TX^{\perp_{\pi}})+rk(TX^{\perp_{\pi}}\cap TX)\\
			&\hspace{0.5cm}-\big(\dim M-\dim X -rk(TX^{\perp_{\pi}})\big)\\
			&=rk(TX)+rk(TX^{\perp_{\pi}})-rk(TX+TX^{\perp_{\pi}})+rk(TX^{\perp_{\pi}}\cap TX)\\
			&=2 rk(TX^{\perp_{\pi}}\cap TX).
		\end{align*}
		
		\noindent
		\underline{Step 2:} $\big(\pi^{\sharp}\big((G+H)^{\circ}\big),\omega\big)$ is a symplectic vector bundle, where
		\[
		\omega(\pi^{\sharp}(\alpha),\pi^{\sharp}(\beta)):=\pi(\alpha,\beta).
		\]
		\noindent
		We first show that $\pi^{\sharp}\big((G+H)^{\circ}\big)\cap (G+H)=\{0\}$. Assume that $\gamma\in(G+H)^{\circ}$ is such that $\pi^{\sharp}(\gamma)=g+h\in G+H$. Since $h\in TX^{\perp_{\pi}}$, we can write $h=\pi^{\sharp}(\beta)$ for some $\beta\in TX^{\circ}$, and we obtain that $\pi^{\sharp}(\gamma-\beta)=g\in TX$. The exact sequence \eqref{sequence} then implies that $\gamma-\beta\in(TX^{\perp_{\pi}})^{\circ}$, and therefore $\gamma\in TX^{\circ}+(TX^{\perp_{\pi}})^{\circ}=(TX\cap TX^{\perp_{\pi}})^{\circ}$. Hence,
		\[
		\gamma\in(TX\cap TX^{\perp_{\pi}})^{\circ}\cap (G+H)^{\circ}=(TX+TX^{\perp_{\pi}})^{\circ}=TX^{\circ}\cap (TX^{\perp_{\pi}})^{\circ},
		\]
		using \eqref{prep} in the first equality. This implies that $\pi^{\sharp}(\gamma)\in TX^{\perp_{\pi}}\cap TX$, so we obtain that $\pi^{\sharp}(\gamma)\in(TX^{\perp_{\pi}}\cap TX)\cap(G+H)=\{0\}$. This shows that $\pi^{\sharp}\big((G+H)^{\circ}\big)\cap (G+H)=\{0\}$.
		
		It now follows that $\omega$ is non-degenerate: if $\pi^{\sharp}(\alpha)\in\ker\omega$ for $\alpha\in(G+H)^{\circ}$, then for all $\beta\in(G+H)^{\circ}$ we get $\langle\pi^{\sharp}(\alpha),\beta\rangle=0$, which shows that $\pi^{\sharp}(\alpha)\in G+H$. By what we just proved, we then get $\pi^{\sharp}(\alpha)\in\pi^{\sharp}\big((G+H)^{\circ}\big)\cap(G+H)=\{0\}$, which shows that $\omega$ is non-degenerate.
		
		\vspace{0.1cm}
		\noindent
		\underline{Step 3:} $TX^{\perp_{\pi}}\cap TX\subset\big(\pi^{\sharp}\big((G+H)^{\circ}\big),\omega\big)$ is a Lagrangian subbundle.
		
		\vspace{0.1cm}
		\noindent
		Since $G+H\subset TX+TX^{\perp_{\pi}}$, we have $(TX+TX^{\perp_{\pi}})^{\circ}\subset (G+H)^{\circ}$ and therefore
		\[
		TX^{\perp_{\pi}}\cap TX=\pi^{\sharp}\big(TX^{\circ}\cap(TX^{\perp_{\pi}})^{\circ}\big)=\pi^{\sharp}\big((TX+TX^{\perp_{\pi}})^{\circ}\big)\subset\pi^{\sharp}\big((G+H)^{\circ}\big).
		\]
		By Step 1, we know that the rank of $\pi^{\sharp}\big((G+H)^{\circ}\big)$ is twice the rank of $TX^{\perp_{\pi}}\cap TX$, so we only have to check that $TX^{\perp_{\pi}}\cap TX\subset\left(\pi^{\sharp}\big((G+H)^{\circ}\big),\omega\right)$ is an isotropic subbundle. This is clearly the case, for if $\alpha,\beta\in TX^{\circ}\cap(TX^{\perp_{\pi}})^{\circ}$ then
		\[
		\omega\big(\pi^{\sharp}(\alpha),\pi^{\sharp}(\beta)\big)=\langle \pi^{\sharp}(\alpha),\beta\rangle=0.
		\]
		Here we use that $\pi^{\sharp}(\alpha)\in TX$ since $\alpha\in \big(TX^{\perp_{\pi}}\big)^{\circ}$, and that $\beta\in TX^{\circ}$.
		
		\vspace{0.1cm}
		\noindent
		\underline{Step 4:} Let $C\subset\big(\pi^{\sharp}\big((G+H)^{\circ}\big),\omega\big)$ be a Lagrangian complement of $TX^{\perp_{\pi}}\cap TX$, and choose \hspace*{1.25cm}any subbundle $Y\subset TM|_{X}$ such that 
		\[
		TM|_{X}=(TX^{\perp_{\pi}}\cap TX)\oplus(H\oplus G\oplus C\oplus Y).
		\]
		\hspace*{1.25cm}Then the subbundle $W_{G,H}:=G\oplus C\oplus Y$ satisfies the criteria.
		
		\vspace{0.1cm}
		\noindent
		If $\alpha\in(H+G+C+Y)^{\circ}$, then $\alpha\in(G+H)^{\circ}$ and $\alpha\in C^{\circ}$. So for all $c\in C$, we get
		\[
		0=\langle \alpha,c\rangle=\omega\big(c,\pi^{\sharp}(\alpha)\big),
		\]
		which implies that $\pi^{\sharp}(\alpha)\in C^{\perp_{\omega}}=C\subset G+C+Y$. Therefore, $\pi^{\sharp}\big((H+W_{G,H})^{\circ}\big)\subset W_{G,H}$. The fact that $W_{G,H}\cap TX=G$ follows immediately from the decomposition
		\[
		TM|_{X}=(TX^{\perp_{\pi}}\cap TX)\oplus H\oplus W_{G,H}=TX\oplus H\oplus C\oplus Y.\qedhere
		\] 	
	\end{proof}
	
	Lemma \ref{compprep} implies that there is a splitting $TM|_{X}=(TX^{\perp_{\pi}}\cap TX)\oplus H\oplus W$, where 
	\[
	TX^{\perp_{\pi}}=(TX^{\perp_{\pi}}\cap TX)\oplus H\hspace{0.5cm}\text{and}\hspace{0.5cm}\pi^{\sharp}\big((H+W)^{\circ}\big)\subset W.
	\]
	Since $\pi\big((H+W)^{\circ},W^{\circ}\big)=0$, a local model for the local Poisson saturation of $X$ defined in terms of the complement $W$ can be constructed out of the data
	\begin{equation}\label{eq:quadruple}
		\Big(H,W,i^{!}L_{\pi},\big(W^{\circ}/(H+W)^{\circ},\pi\big)\Big).
	\end{equation}
	Interpreting the vector bundle $W^{\circ}/(H+W)^{\circ}$ as a well-defined version of the ``quotient'' $(TX^{\perp_{\pi}})^{*}/(TX^{\perp_{\pi}}\cap TX)^{*}$, we regard this fact as a Poisson analog of Marle's theorem.
	
	\begin{cor}[Poisson version of Marle's theorem]\label{marlepoisson}
		If $X\subset(M,\pi)$ is a coregular pre-Poisson submanifold, then a quadruple as in \eqref{eq:quadruple} determines a neighborhood of $X$ in its local Poisson saturation, up to Poisson diffeomorphism.
	\end{cor}
	
	The corollary shows that the local Poisson saturation of a coregular pre-Poisson submanifold is determined by less data than that of a general coregular submanifold, since it uses $\pi$ on a quotient of $W^{\circ}$ rather than on all of $W^{\circ}$.  The exception are those pre-Poisson submanifolds $X$ for which $TX^{\perp_{\pi}}\cap TX=0$; these are the coregular Poisson-Dirac submanifolds of $(M,\pi)$ (see \cite[\S 8.2]{Apaths} or \cite[\S 8.3]{bookpoiss}). They are studied in the recent work \cite{coregular}.
	
	
	\begin{remark}
		For the classes of coregular submanifolds $X\subset (M,\pi)$ considered in this section, we summarize loosely the data that determine the local Poisson saturation $(P,\pi_P)$ near $X$.
		\begin{center}
			\begin{tabular}{||c| c ||} 
				\hline
				Type of submanifold & $(P,\pi_P)$ locally determined by \\ [0.5ex] 
				\hline\hline
				$X\subset (M,\pi)$ Poisson transversal & $i^{!}L_{\pi}$ and $\pi|_{(TX^{\perp_{\pi}})^{*}}$\\
				\hline
				$X\subset (M,\pi)$ coregular coisotropic &  $i^{!}L_{\pi}$\\
				\hline
				$X\subset (M,\pi)$ coregular pre-Poisson &  $i^{!}L_{\pi}$ and $\pi|_{(TX^{\perp_{\pi}})^{*}/(TX^{\perp_{\pi}}\cap TX)^{*}}$\\
				\hline
			\end{tabular}
		\end{center}
	\end{remark}

	\section{Coisotropic embeddings of Dirac manifolds in Poisson manifolds}\label{sec:fifth}
	
	As an application of Corollary \ref{coiso}, we look at the following question, which was considered by Cattaneo and Zambon \cite{marco} and by Wade \cite{wade}: Given a Dirac manifold $(X,L)$, when can it be embedded coisotropically into a Poisson manifold $(M,\pi)$? That is, when does there exist an embedding $i:X\hookrightarrow (M,\pi)$ such that $i^{!}L_{\pi}=L$ and $i(X)$ is coisotropic in $(M,\pi)$? Moreover, to what extent is such an embedding unique?
	
	The question on the existence of coisotropic embeddings $(X,L)\hookrightarrow (M,\pi)$ is settled in \cite[Theorem 8.1]{marco}: such an embedding exists exactly when $L\cap TX$ has constant rank. The construction of $(M,\pi)$ in that case is carried out as follows: a choice of complement $V$ to $L\cap TX$ in $TX$ gives an inclusion $j:(L\cap TX)^{*}\hookrightarrow T^{*}X$, one takes $M$ to be the total space of the vector bundle $\mathrm{pr}:(L\cap TX)^{*}\rightarrow X$ and one shows that the Dirac structure
	$
	(\mathrm{pr}^{!}L)^{j^{*}\omega_{can}}
	$
	on $M$ is in fact Poisson on a neighborhood of $X\subset M$. 
	A different proof of the existence result is given in \cite[Theorem 4.1]{wade}.
	
	The question on the uniqueness of coisotropic embeddings $(X,L)\hookrightarrow (M,\pi)$ is still open. In \cite{wade}, it is claimed (without proof) that uniqueness can be obtained if $L\cap TX$ defines a simple foliation on $X$. In \cite{marco} it is conjectured that, if $(X,L)$ is embedded coisotropically  in two different Poisson manifolds, then these must be neighborhood equivalent around $X$, provided that they are of minimal dimension $\dim X + rk(L\cap TX)$. However, a proof of this uniqueness statement is only given under the additional regularity assumption that the presymplectic leaves of $(X,L)$ have constant dimension \cite[Proposition 9.4]{marco}. 
	
	We now show that this extra assumption can be dropped. Using Corollary \ref{coiso}, we prove that the model $\big(U,(\mathrm{pr}^{!}L)^{j^{*}\omega_{can}}\big)$ constructed in \cite{marco} is minimal, thereby obtaining the uniqueness result in full generality. In the proof below, given an embedding $i:X\hookrightarrow(M,\pi)$, we may assume that it is the inclusion map by identifying $X$ with $i(X)$.
	
	\begin{prop}
		Let $(X,L)$ be a Dirac manifold for which $L\cap TX$ has constant rank, and denote by $\mathrm{pr}:(L\cap TX)^{*}\rightarrow X$ the bundle projection.
		\begin{enumerate}[i)]
			\item Any coisotropic embedding $i:(X,L)\hookrightarrow (M,\pi)$ into a Poisson manifold $(M,\pi)$ factors through the local model $\big(U,(\mathrm{pr}^{!}L)^{j^{*}\omega_{can}}\big)$. That is, we have a diagram
			\[
			\begin{tikzcd}[row sep= large, column sep=large]
				(X,L)\arrow[r,"i"]\arrow[d,hookrightarrow] & (M,\pi)\\
				\big(U,(\mathrm{pr}^{!}L)^{j^{*}\omega_{can}}\big)\arrow[ru,hookrightarrow,"\psi"] & 
			\end{tikzcd},
			\]
			where $\psi:\big(U,(\mathrm{pr}^{!}L)^{j^{*}\omega_{can}}\big)\hookrightarrow (M,\pi)$ is a Poisson embedding.
			\item In particular, if $(M_{1},\pi_{1})$ and $(M_{2},\pi_{2})$ are Poisson manifolds of minimal dimension $\dim X + rk(L\cap TX)$ in which $(X,L)$ embeds coisotropically, then $(M_{1},\pi_{1})$ and $(M_{2},\pi_{2})$ are Poisson diffeomorphic around $X$.	
		\end{enumerate}
	\end{prop}
	\begin{proof}
		\begin{enumerate}[i)]
			\item The assumptions imply that $X\subset (M,\pi)$ is a coregular coisotropic submanifold, since
			\begin{equation}\label{perp}
				TX^{\perp_{\pi}}=\pi^{\sharp}(TX^{\circ})=(i^{!}L_{\pi})\cap TX=L\cap TX.
			\end{equation}
			Denote by $(P,\pi_{P})$ the local Poisson saturation of $X\subset (M,\pi)$. By Corollary \ref{coiso}, there is a neighborhood $U\subset (L\cap TX)^{*}$ of $X$ and a Poisson embedding 
			\[
			\phi:\big(U,(\mathrm{pr}^{!}L)^{j^{*}\omega_{can}}\big)\rightarrow (P,\pi_{P}).
			\]
			Since $(P,\pi_{P})$ is an embedded submanifold of $(M,\pi)$, this proves the statement.
			\item By what we just proved, there exist a neighborhood $U\subset (L\cap TX)^{*}$ of $X$ and two Poisson embeddings
			\begin{align*}
				&\phi_{1}:\left(U,(\mathrm{pr}^{!}L)^{j^{*}\omega_{can}}\right)\rightarrow (P_{1},\pi_{P_{1}}),\\
				&\phi_{2}:\left(U,(\mathrm{pr}^{!}L)^{j^{*}\omega_{can}}\right)\rightarrow (P_{2},\pi_{P_{2}}),
			\end{align*}
			where $(P_{1},\pi_{P_{1}})$ and $(P_{2},\pi_{P_{2}})$ denote the local Poisson saturations of $X$ in $(M_{1},\pi_{1})$ and $(M_{2},\pi_{2})$, respectively. The assumption implies that, for $l=1,2$:
			\[
			\dim P_{l}=\dim TX^{\perp_{\pi_l}}=\dim X + rk(L\cap TX)=\dim M_{l},
			\]
			where we used \eqref{perp}. Since $P_{l}\subset M_{l}$ is an embedded submanifold, this shows that $P_{l}\subset M_{l}$ is a neighborhood of $X$, for $l=1,2$. So the composition $\phi_{2}\circ\phi_{1}^{-1}$ is a Poisson diffeomorphism between neighborhoods of $X$ in $(M_{1},\pi_{1})$ and $(M_{2},\pi_{2})$. 
		\end{enumerate}
	\end{proof}
	
	\section{Coregular submanifolds in Dirac geometry}\label{sec:sixth}
	We now discuss how the results that we obtained in Sections \ref{sec:first}, \ref{sec:second} and \ref{sec:third} can be generalized to the setting of Dirac manifolds. The relevant tools are developed in \cite{dirac}, from which we adopt the terminology and notation. For background on Dirac geometry, see e.g. \cite{bursztyn}. 
	
	\begin{defi}
		We call an embedded submanifold $X$ of a Dirac manifold $(M,L)$ \textbf{coregular} if the map $\overline{\mathrm{pr}_{T}}:L|_{X}\rightarrow TM|_{X}/TX$, which is obtained composing the anchor $\mathrm{pr}_{T}:L\rightarrow TM$ with the projection to the normal bundle, has constant rank.
	\end{defi}
	
	Given any submanifold $i:X\hookrightarrow(M,L)$, we have at points $x\in X$ that
	\[
	\overline{\mathrm{pr}_{T}}(L_x)=\frac{\mathrm{pr}_{T}(L_x)+T_{x}X}{T_{x}X},
	\]
	and therefore
	\begin{align*}
		X\subset (M,L)\ \text{is coregular}&\Leftrightarrow \mathrm{pr}_{T}(L)+TX\ \text{has constant rank}\\
		&\Leftrightarrow \ker((\mathrm{d}i)^{*})\cap L\ \text{has constant rank},
	\end{align*}
	using that $\ker((\mathrm{d}i)^{*})\cap L=(\mathrm{pr}_{T}(L)+TX)^{\circ}$. In particular, the Dirac structure $L$ automatically induces a Dirac structure on a coregular submanifold $X\subset(M,L)$ \cite[Prop. 1.10]{bursztyn}.
	
	\vspace{0.2cm}
	
	We recall some results about sprays and dual pairs in Dirac geometry \cite{dirac}.
	
	\begin{defi}
		Let $L\subset TM\oplus T^{*}M$ be a Dirac structure on $M$, and let $\textbf{s}:L\rightarrow M$ denote the bundle projection. A \textbf{spray} for $L$ is a vector field $\mathcal{V}\in\mathfrak{X}(L)$ satisfying
		\begin{enumerate}[i)]
			\item $\mathrm{d}\textbf{s}(\mathcal{V}_{a})=\mathrm{pr}_{T}(a)$ for all $a\in L$,
			\item $m_{t}^{*}\mathcal{V}=t\mathcal{V}$, where $m_{t}:L\rightarrow L$ denotes fiberwise multiplication by $t\neq 0$.
		\end{enumerate}
	\end{defi}
	
	Sprays exist on any Dirac structure. Condition ii) implies that the spray $\mathcal{V}$ vanishes along the zero section $M\subset L$, and therefore there exists a neighborhood $\Sigma\subset L$ of $M$ on which the flow $\varphi_{\epsilon}$ of $\mathcal{V}$ is defined for all times $\epsilon\in[0,1]$. We can then define the \textbf{spray exponential} associated with $\mathcal{V}$ as
	\[
	\exp_{\mathcal{V}}:\Sigma\rightarrow M:a\mapsto\textbf{s}(\varphi_{1}(a)).
	\]
	Moreover, this neighborhood $\Sigma\subset L$ supports a two-form $\omega$ defined by
	\begin{equation}\label{eq:sprayform}
		\omega:=\int_{0}^{1}\varphi_{\epsilon}^{*}\big((\mathrm{pr}_{T^{*}})^{*}\omega_{can}\big)\mathrm{d}\epsilon,
	\end{equation}
	where $\mathrm{pr}_{T^{*}}:L\rightarrow T^{*}M$ is the projection and $\omega_{can}$ is the canonical symplectic form on $T^{*}M$. It is proved in \cite{dirac} that, shrinking $\Sigma\subset L$ if necessary, these data fit into a \textbf{Dirac dual pair}:
	\begin{equation}\label{dualpair}
		\begin{tikzcd}
			(M,L) & (\Sigma,\text{Gr}(\omega))\arrow[l,"\textbf{s}", swap]\arrow{r}{\exp_{\mathcal{V}}} & (M,-L).
		\end{tikzcd}
	\end{equation}
	This means that both legs in the diagram \eqref{dualpair} are surjective, forward Dirac submersions, and we have the additional requirements that $\omega(V,W)=0$ and $V\cap K\cap W=0$, where $V=\ker \mathrm{d}\textbf{s}, W=\ker \mathrm{d}\exp_{\mathcal{V}}$ and $K=\ker\omega$.	
	
	We need the following lemma, which is a Dirac substitute for Lemma \ref{realization}. The statement is not exactly the Dirac analog of Lemma \ref{realization}; we address this in Remark \ref{presym} below.
	
	\begin{lemma}\label{rank}
		Consider a Dirac dual pair
		\begin{equation*}
			\begin{tikzcd}[row sep=large, column sep=large]
				(M_0,L_0) & (\Sigma,\text{Gr}(\omega))\arrow[l,"\textbf{s}", swap]\arrow{r}{\textbf{t}} & (M_1,-L_1),
			\end{tikzcd}
		\end{equation*}
		and let $X\subset(M_0,L_0)$ be a coregular submanifold. We denote $V:=\ker \mathrm{d}\textbf{s}$, $W:=\ker \mathrm{d}\textbf{t}$ and $K:=\ker\omega$. Then $W\cap \mathrm{d}\textbf{s}^{-1}(TX)$ has constant rank, equal to the rank of $\mathrm{pr}_{T}^{-1}(TX)\subset L_0|_{X}$.	
	\end{lemma}	
	\begin{proof}
		Consider the following diagram of vector bundle maps:
		\begin{equation}\label{diag2}
			\begin{tikzcd}[row sep=large, column sep=large]
				W|_{\textbf{s}^{-1}(X)}\arrow{r}{\overline{\mathrm{d}\textbf{s}}}\arrow[d,"R_{\omega}",swap] & TM_0|_{X}/TX\\
				R_{\omega}(W|_{\textbf{s}^{-1}(X)})\arrow[r,"\psi",swap]&L_0|_{X}\arrow[u,"\overline{\mathrm{pr}_{T}}",swap]
			\end{tikzcd}.
		\end{equation}
		Here $R_{\omega}$ is an injective bundle map defined by $R_{\omega}:W\rightarrow T\Sigma\oplus T^{*}\Sigma:w\mapsto w+\iota_{w}\omega$. The map $\psi:R_{\omega}(W)\rightarrow L_0$ is defined by setting $\psi(w+\iota_{w}\omega):=\mathrm{d}\textbf{s}(w)+\beta$, where $\beta$ is uniquely determined by the relation $\mathrm{d}\textbf{s}^{*}(\beta)=\iota_{w}\omega$. Note that $\psi$ is well-defined: existence of $\beta$ follows from the fact that $\omega(V,W)=0$, and $\beta$ is unique since $\textbf{s}$ is a submersion. Since the map $\textbf{s}:(\Sigma,\text{Gr}(\omega))\rightarrow (M_0,L_0)$ is forward Dirac, $\psi(w+\iota_{w}\omega)=\mathrm{d}\textbf{s}(w)+\beta$ is contained in $L_0$. 
		
		Moreover, we claim that the map $\psi$ is an isomorphism. To see that $\psi$ is injective, assume that $\psi(w+\iota_{w}\omega)=\mathrm{d}\textbf{s}(w)+\beta=0$ for some $w\in W$. Then $\beta=0$, and therefore $\iota_{w}\omega=\mathrm{d}\textbf{s}^{*}(\beta)=0$, so that $w\in W\cap K$. But also $\mathrm{d}\textbf{s}(w)=0$, so that $w\in V$. Hence $w\in V\cap K\cap W=0$, which shows that $\psi$ is injective. Since the rank of $R_{\omega}(W)$ is given by
		\[
		rk(R_{\omega}(W))=rk(W)=\dim\Sigma-\dim M_1=\dim M_0=rk(L_0),
		\]
		it follows that $\psi:R_{\omega}(W)\rightarrow L_0$ is a vector bundle isomorphism. Since the diagram \eqref{diag2} commutes, it follows that 
		\begin{align*}
			rk\big(\overline{\mathrm{d}\textbf{s}}:W|_{\textbf{s}^{-1}(X)}\rightarrow TM_0|_{X}/TX\big)&=rk\big(\overline{\mathrm{pr}_{T}}:L_0|_{X}\rightarrow TM_0|_{X}/TX\big)\\
			&=\dim M_0-rk\big(\mathrm{pr}_{T}^{-1}(TX)\big).
		\end{align*}
		Consequently, we obtain that
		\[
		rk\big(W\cap \mathrm{d}\textbf{s}^{-1}(TX)\big)=rk(W)-\dim M_0 +rk\big(\mathrm{pr}_{T}^{-1}(TX)\big)=rk\big(\mathrm{pr}_{T}^{-1}(TX)\big),
		\]
		which finishes the proof of the lemma.
	\end{proof}
	
	\begin{remark}\label{presym}
		For completeness, we state here the Dirac geometric analog of Lemma \ref{realization}. Recall that a forward Dirac map $\varphi:(M_{0},L_{0})\rightarrow (M_{1},L_{1})$ is \textbf{strong} if $L_{0}\cap\ker \mathrm{d}\varphi=0$. When $L_0$ is the graph of a closed $2$-form, then the map $\varphi$ is called a \textbf{presymplectic realization} of $(M_{1},L_{1})$. One can show that the following is true:
		
		``\textit{Let $\textbf{s}:(\Sigma,\text{Gr}(\omega))\rightarrow (M,L)$ be a strong forward Dirac submersion, and assume that $X\subset(M,L)$ is a coregular submanifold. If $V:=\ker \mathrm{d}\textbf{s}$, then $V^{\perp_{\omega}}\cap  \mathrm{d}\textbf{s}^{-1}(TX)$ has constant rank, equal to the rank of $\mathrm{pr}_{T}^{-1}(TX)$}.''
		
		We won't address this in more detail, since we want to use the legs of the diagram \eqref{dualpair} and these are in general not presymplectic realizations. Indeed, using expressions for $\omega|_{M}$ that appear in \cite{dirac}, one can check that
		\begin{align*}
			&(\text{Gr}(\omega)\cap\ker \mathrm{d}\textbf{s})|_{M}=0\oplus L\cap TM\subset TM\oplus L,\\
			&(\text{Gr}(\omega)\cap\ker \mathrm{d}\exp_{\mathcal{V}})|_{M}=\{(-v,v):v\in L\cap TM\}\subset TM\oplus L,
		\end{align*}
		so that both legs are presymplectic realizations only when the Dirac structure $L$ is Poisson. In that case, $\omega$ is non-degenerate along $M\subset\Sigma$, so that shrinking $\Sigma$ if necessary, the diagram \eqref{dualpair} is a full dual pair. In particular, the legs of the diagram \eqref{dualpair} are symplectic realizations.
	\end{remark}	
	
	We obtain the following generalization of Theorem \ref{satsmooth}. 
	
	\begin{thm}\label{submanifold}
		Let $X\subset(M,L)$ be a coregular submanifold. 
		\begin{enumerate}
			\item There exists an embedded invariant submanifold $(P,L_P)\subset (M,L)$ containing $X$ that lies inside the saturation $Sat(X)$. 
			\item Shrinking $P$ if necessary, there exists a neighborhood $U$ of $X$ in $M$ such that $(P,L_{P})$ is the saturation of $X$ in $(U,L|_{U})$.
		\end{enumerate}
	\end{thm}
	\begin{proof}
		The proof is divided into four steps, just like the proof of Theorem \ref{satsmooth}. 
		
		\vspace{0.2cm}
		\noindent
		\underline{Step 1:} Construction of the submanifold $P\subset M$.
		
		\vspace{0.1cm}
		\noindent
		Choose a spray $\mathcal{V}\in\mathfrak{X}(L)$ and denote by $\exp_{\mathcal{V}}:\Sigma\subset L\rightarrow M$ the corresponding spray exponential. Let $\textbf{s}:L\rightarrow M$ denote the bundle projection. Note that $\exp_{\mathcal{V}}(a)$ and $\textbf{s}(a)$ lie in the same presymplectic leaf of $(M,L)$, for all $a\in L$. Indeed, the path $t\mapsto\varphi_{t}(a)$ is an $A$-path for the Lie algebroid $A=\big(L,[\![\cdot,\cdot]\!],\mathrm{pr}_{T}\big)$, covering the path $t\mapsto \textbf{s}(\varphi_{t}(a))$ which connects $\textbf{s}(a)$ with $\exp_{\mathcal{V}}(a)$. In particular, we have that $\exp_{\mathcal{V}}(\Sigma|_{X})\subset Sat(X)$.
		
		Since $X\subset(M,L)$ is coregular, we have that $\mathrm{pr}_{T}^{-1}(TX)$ is a subbundle of $L|_{X}$, being the kernel of the constant rank bundle map $\overline{\mathrm{pr}_{T}}:L|_{X}\rightarrow TM|_{X}/TX$. Choose a complement $L|_{X}=\mathrm{pr}_{T}^{-1}(TX)\oplus C$ and consider the restriction $\exp_{\mathcal{V}}:C\cap\Sigma\rightarrow M$. It fixes points of $X$, and its differential along $X$ reads \cite[Lemma 7]{dirac}:
		\[
		\mathrm{d}\exp_{\mathcal{V}}:T_{x}X\oplus C_{x}\rightarrow T_{x}M:(u,a)\mapsto u+\mathrm{pr}_{T}(a).
		\]
		This map is injective and therefore, shrinking $\Sigma$ if necessary, the map $\exp_{\mathcal{V}}:C\cap\Sigma\rightarrow M$ is an embedding by Prop. \ref{embedding}. We set $P:=\exp_{\mathcal{V}}(C\cap\Sigma)$.
		
		\vspace{0.2cm}
		\noindent
		\underline{Step 2:} Shrinking $\Sigma$ if necessary, we have that $P=\exp_{\mathcal{V}}(\Sigma|_{X})$.
		
		\vspace{0.1cm}
		\noindent
		It is enough to show that the restriction of $\exp_{\mathcal{V}}$ to $\Sigma|_{X}$ has constant rank, equal to the rank of $\exp_{\mathcal{V}}|_{C\cap\Sigma}$. To see this, we apply Lemma \ref{rank} to the self-dual pair \eqref{dualpair}, and we obtain that
		\[
		\ker\big(\mathrm{d}(\exp_{\mathcal{V}}|_{\Sigma|_{X}})\big)=\ker(\mathrm{d}\exp_{\mathcal{V}})\cap \mathrm{d}\textbf{s}^{-1}(TX)
		\]
		has constant rank, equal to the rank of $\mathrm{pr}_{T}^{-1}(TX)\subset L|_{X}$. This implies that the rank of $\exp_{\mathcal{V}}|_{\Sigma|_{X}}$ is constant, equal to 
		\begin{align*}
			rk\big(\exp_{\mathcal{V}}|_{\Sigma|_{X}}\big)&=\dim X+rk(L) -rk\big(\mathrm{pr}_{T}^{-1}(TX)\big)\\
			&=\dim X+ rk(C)\\
			&=rk\big(\exp_{\mathcal{V}}|_{C\cap\Sigma}\big).
		\end{align*}
		
		\vspace{0.1cm}
		\noindent
		\underline{Step 3:} The submanifold $P\subset(M,L)$ is invariant.
		
		\vspace{0.1cm}
		\noindent
		
		We have to check that the characteristic distribution $\mathrm{pr}_{T}(L)$ of $L$  is tangent to $P$, i.e. that $\mathrm{pr}_{T}\big(L_{\exp_{\mathcal{V}}(a)}\big)\subset(\mathrm{d}\exp_{\mathcal{V}})_{a}(T_{a}\Sigma|_{X})$ for all $a\in\Sigma|_{X}$. 
		We will first show that 
		\[
		\mathrm{pr}_{T}\big(L_{\exp_{\mathcal{V}}(a)}\big)=(\mathrm{d}\exp_{\mathcal{V}})_{a}(W^{\perp_{\omega}}),
		\] 
		where $W$ denotes $\ker \mathrm{d}\exp_{\mathcal{V}}$ as before. To see this, first pick $u+\xi\in L_{\exp_{\mathcal{V}}(a)}$. Then $u-\xi\in -L$, and since the map $\exp_{\mathcal{V}}:(\Sigma,\text{Gr}(\omega))\rightarrow(M,-L)$ is forward Dirac, there exists $v\in T_{a}\Sigma$ such that $v+\iota_{v}\omega$ is $\exp_{\mathcal{V}}$-related with $u-\xi$, i.e.
		\[
		\begin{cases}
			\iota_{v}\omega=(\mathrm{d}\exp_{\mathcal{V}})_{a}^{*}(-\xi),\\
			u=(\mathrm{d}\exp_{\mathcal{V}})_{a}(v).
		\end{cases}
		\]
		This implies that $v\in W^{\perp_{\omega}}$, so $\mathrm{pr}_{T}(u+\xi)=u=(\mathrm{d}\exp_{\mathcal{V}})_{a}(v)$ is contained in $(\mathrm{d}\exp_{\mathcal{V}})_{a}(W^{\perp_{\omega}})$. Conversely, assume $v\in T_{a}\Sigma$ lies in $W^{\perp_{\omega}}$. Then $\iota_{v}\omega=(\mathrm{d}\exp_{\mathcal{V}})_{a}^{*}(\xi)$ for some $\xi\in T_{\exp_{\mathcal{V}}(a)}^{*}M$. This implies that $(\mathrm{d}\exp_{\mathcal{V}})_{a}(v)+\xi$ is $\exp_{\mathcal{V}}$-related with $v+\iota_{v}\omega\in\text{Gr}(\omega)$, and since the map $\exp_{\mathcal{V}}:(\Sigma,\text{Gr}(\omega))\rightarrow(M,-L)$ is forward Dirac, we get that $(\mathrm{d}\exp_{\mathcal{V}})_{a}(v)+\xi\in -L$, i.e. $(\mathrm{d}\exp_{\mathcal{V}})_{a}(v)-\xi\in L$. It follows that $(\mathrm{d}\exp_{\mathcal{V}})_{a}(v)\in \mathrm{pr}_{T}\big(L_{\exp_{\mathcal{V}}(a)}\big)$.
		
		Consequently, we obtain that
		\begin{align}\label{family}
			\mathrm{pr}_{T}\big(L_{\exp_{\mathcal{V}}(a)}\big)&=(\mathrm{d}\exp_{\mathcal{V}})_{a}\big(W^{\perp_{\omega}}\big)\nonumber\\
			&=(\mathrm{d}\exp_{\mathcal{V}})_{a}\big(V+W\cap K\big)\nonumber\\
			&\subset(\mathrm{d}\exp_{\mathcal{V}})_{a}\big(\mathrm{d}\textbf{s}^{-1}(TX)\big)\nonumber\\
			&=(\mathrm{d}\exp_{\mathcal{V}})_{a}(T_{a}\Sigma|_{X}),
		\end{align}
		where the second equality uses \cite[Lemma 3]{dirac}, and the third equality holds because $W=\ker \mathrm{d}\exp_{\mathcal{V}}$ and $V=\ker \mathrm{d}\textbf{s}\subset \mathrm{d}\textbf{s}^{-1}(TX)$. This proves Step 3.
		
		\vspace{0.2cm}
		\noindent
		\underline{Step 4:} Construction of the neighborhood $U$ of $X$.
		
		\vspace{0.1cm}
		\noindent
		The proof is completely analogous to the proof of Step 4 in Theorem \ref{satsmooth}. We want to extend the map $\exp_{\mathcal{V}}:C\cap\Sigma\rightarrow M$ to a local diffeomorphism. To do so, we choose a complement
		\[
		TM|_{X}=TX\oplus\left(\mathrm{pr}_{T}(C)\oplus E\right)
		\]
		and a linear connection $\nabla$ on $TM$. We obtain a map
		\[
		\psi:O\subset(C\oplus E)\rightarrow M:(a,e)\mapsto\exp_{\nabla}\left(Tr_{\exp_{\mathcal{V}}(ta)}e\right),
		\]
		which is a diffeomorphism onto a neighborhood of $X$.
		Here $O$ is a suitable convex neighborhood of the zero section, and $Tr_{\exp_{\mathcal{V}}(ta)}$ denotes parallel transport along the curve $t\mapsto\exp_{\mathcal{V}}(ta)$ for $t\in[0,1]$. Note that $\psi(a,0)=\exp_{\mathcal{V}}(a)$, so shrinking $P$, we can assume that $P=\psi(O\cap (C\oplus\{0\}))$. Setting $U:=\psi(O)$ finishes the proof.
	\end{proof}
	
	In the following, we denote by $(P,L_P)$ the Dirac manifold constructed in Thm. \ref{submanifold}; we refer to it as the \textbf{local Dirac saturation} of $X$. Since $X$ is a Dirac transversal in $(P,L_P)$, the normal form theorem around Dirac transversals \cite{eulerlike}, \cite{dirac} gives a normal form for the local Dirac saturation around $X$. We will reprove this result, continuing the argument from Theorem \ref{submanifold}. We need the following Dirac version of Lemma \ref{diracpair}. 
	

	\begin{lemma}
		Let $i:X\hookrightarrow (M,L)$ be a coregular submanifold with local Dirac saturation $(P,L_P)$. Then the following is a weak Dirac dual pair, in the sense of \cite{dirac}:
		\begin{equation}\label{weakpair}
			\begin{tikzcd}
				(X,i^{!}L)&\big(\Sigma|_{X},\text{Gr}(\omega|_{X})\big)\arrow{r}{\exp_{\mathcal{V}}}\arrow[l,"\mathbf{s}",swap]&(P,-L_{P}).
			\end{tikzcd}
		\end{equation}
		This means that $\omega|_{X}$ is a closed two-form on $\Sigma|_{X}$, that $\mathbf{s}$ and $\exp_{\mathcal{V}}$ are surjective forward Dirac submersions, and that
		\begin{align}
			&\omega|_{X}(S_{1},S_{2})=0,\label{prop1}\\
			&rk(S_{1}\cap \widetilde{K}\cap S_{2})=\dim \Sigma|_{X}-\dim X-\dim P,\label{prop2}
		\end{align}
		where $S_{1}:=\ker \mathrm{d}\mathbf{s},S_{2}:=\ker\mathrm{d}\exp_{\mathcal{V}}$ and $\widetilde{K}:=\ker(\omega|_{X})$.
	\end{lemma}
	\begin{proof}
		The only non-trivial part is that equality \eqref{prop2} holds. The other claims are proved exactly like in Lemma \ref{diracpair}, so we don't address them here.
		
		To prove \eqref{prop2}, note that $S_{1}\cap \widetilde{K}\cap S_{2}=V\cap (\mathrm{d}\textbf{s}^{-1}(TX))^{\perp_{\omega}}\cap W$, where $V,W$ are the vertical distributions of the original dual pair \eqref{dualpair}. Note that for any subspace $U_a\subset(T_{a}\Sigma,\omega_{a})$, we have 
		\[
		\dim (U_a^{\perp_{\omega}})=\dim (T_a\Sigma)-\dim (U_a)+\dim(U_a\cap K_a),
		\]
		where $K:=\ker\omega$.
		It follows that a family $U\subset T\Sigma$ of linear subspaces has constant rank if both $U^{\perp_{\omega}}$ and $U\cap K$ have constant rank. On one hand, we have
		\[
		\big(V\cap (\mathrm{d}\textbf{s}^{-1}(TX))^{\perp_{\omega}}\cap W\big)\cap K=0,
		\]
		since $V\cap K\cap W=0$. On the other hand, we consider $\big(V\cap (\mathrm{d}\textbf{s}^{-1}(TX))^{\perp_{\omega}}\cap W\big)^{\perp_{\omega}}$. Using that $K\subset(\mathrm{d}\textbf{s}^{-1}(TX))^{\perp_{\omega}}$, one checks that
		\[
		\big(V\cap (\mathrm{d}\textbf{s}^{-1}(TX))^{\perp_{\omega}}\cap W\big)^{\perp_{\omega}}=(V\cap W)^{\perp_{\omega}}+\big((\mathrm{d}\textbf{s}^{-1}(TX))^{\perp_{\omega}}\big)^{\perp_{\omega}}.
		\]
		Moreover, using that $V$ and $W$ are the vertical distributions of the dual pair \eqref{dualpair}, one proves that $(V\cap W)^{\perp_{\omega}}=V^{\perp_{\omega}}+W^{\perp_{\omega}}$. Altogether, we obtain 
		\begin{align*}
			\big(V\cap (\mathrm{d}\textbf{s}^{-1}(TX))^{\perp_{\omega}}\cap W\big)^{\perp_{\omega}}&=V^{\perp_{\omega}}+\big((\mathrm{d}\textbf{s}^{-1}(TX))^{\perp_{\omega}}\big)^{\perp_{\omega}}+W^{\perp_{\omega}}\\
			&=V^{\perp_{\omega}}+\mathrm{d}\textbf{s}^{-1}(TX)+K+W^{\perp_{\omega}}\\
			&=V^{\perp_{\omega}}+\mathrm{d}\textbf{s}^{-1}(TX)+W^{\perp_{\omega}}\\
			&=W+V\cap K+\mathrm{d}\textbf{s}^{-1}(TX)+V+W\cap K\\
			&=W+\mathrm{d}\textbf{s}^{-1}(TX)+V\\
			&=W+\mathrm{d}\textbf{s}^{-1}(TX).
		\end{align*}
		In the fourth equality, we use \cite[Lemma 3]{dirac}. Using Lemma \ref{rank}, we have now proved that $S_{1}\cap \widetilde{K}\cap S_{2}=V\cap (\mathrm{d}\textbf{s}^{-1}(TX))^{\perp_{\omega}}\cap W$ has constant rank. The rank is given by
		\begin{align*}
			rk\big(V\cap (\mathrm{d}\textbf{s}^{-1}(TX))^{\perp_{\omega}}\cap W\big)&=rk(T\Sigma)-rk(W+\mathrm{d}\textbf{s}^{-1}(TX))\\
			&=rk(T\Sigma)-rk(W)-rk(\mathrm{d}\textbf{s}^{-1}(TX))+rk(W\cap \mathrm{d}\textbf{s}^{-1}(TX))\\
			&=\dim(L) - rk(W)-\dim(X)-rk(V)+rk(\mathrm{pr}_{T}^{-1}(TX))\\
			&=\dim (M)-\dim(\exp_{\mathcal{V}}(\Sigma|_{X}))+rk(L)-rk(V)\\
			&=\dim(\Sigma)-rk(V)-\dim(\exp_{\mathcal{V}}(\Sigma|_{X}))\\
			&=\dim(\Sigma|_{X})-\dim(X)-\dim(\exp_{\mathcal{V}}(\Sigma|_{X})).
		\end{align*}
		This is exactly the rank condition \eqref{prop2}, so the proof is finished.
	\end{proof}
	
	\begin{cor}\label{diractransversal}
		Let $i:X\hookrightarrow(M,L)$ be a coregular submanifold. We choose a complement $L|_{X}=\mathrm{pr}_{T}^{-1}(TX)\oplus C$ and denote by $j:C\hookrightarrow L|_{X}$ the inclusion. The local Dirac saturation $(P,L_P)$ of $X$ is diffeomorphic with
		\[
		\big(C\cap\Sigma,(\textbf{s}^{!}(i^{!}L))^{-j^{*}\omega|_{X}}\big).
		\]
		In particular, $(P,L_P)$ is determined by the pullback Dirac structure $i^{!}L$, up to diffeomorphisms and exact gauge transformations.
	\end{cor}
	\begin{proof}
		Applying \cite[Prop.6]{dirac} to the diagram \eqref{weakpair}, we have the following equality of Dirac structures on $\Sigma|_{X}$:
		\begin{equation*}
			(\textbf{s}^{!}(i^{!}L))^{-\omega|_{X}}=(\exp_{\mathcal{V}})^{!}L_{P}.
		\end{equation*}
		Taking the pullback under the map $j$, which is transverse to this Dirac structure, we obtain
		\[
		(\textbf{s}^{!}(i^{!}L))^{-j^{*}\omega|_{X}}=(\exp_{\mathcal{V}}\circ j)^{!}L_{P},
		\]
		which is an equality of Dirac structures on $C\cap\Sigma$. We showed in Theorem \ref{submanifold} that $\exp_{\mathcal{V}}\circ j$ is a diffeomorphism from $C\cap\Sigma$ onto $P$, which proves the first statement. Moreover, since $-j^{*}\omega|_{X}$ is closed and its pullback to $X\subset C\cap\Sigma$ vanishes, it is exact on a neighborhood of $X$, by the relative Poincar\'e lemma. This implies the second statement of the corollary.
	\end{proof}
	
	\begin{remark}
		As mentioned before, the submanifold $X$ is a Dirac transversal in $(P,L_{P})$. Corollary \ref{diractransversal} agrees with the normal form around Dirac transversals proved in \cite{dirac}, upon identifying the normal bundle $TP|_{X}/TX$ with $C$.
	\end{remark}
	
	\section{Appendix}
	
	We prove a result in differential topology that may be of independent interest. It should be standard, but we could not find a reference in the literature. The statement is well-known under the stronger assumption that the derivative of the map is an isomorphism along the zero section \cite[Lemma 6.1.3]{mukherjee}. Our strategy is to reduce the proof to this case.
	
	\begin{prop}\label{embedding}
		Let $E\rightarrow N$ be a vector bundle, and let $\varphi:E\rightarrow M$ be a smooth map satisfying
		\begin{equation}\label{assumptions}
			\begin{cases}
				\varphi|_{N}\ \text{is an embedding}\\
				(\mathrm{d}\varphi)_{p}\ \text{is injective}\ \forall  p\in N
			\end{cases}.
		\end{equation}
		Then there is a neighborhood $U\subset E$ of $N$ such that $\varphi|_{U}$ is an embedding.
	\end{prop}
	\begin{proof}
		We get a vector subbundle $\mathrm{d}\varphi|_{N}(E)\subset TM|_{\varphi(N)}$ which has trivial intersection with $T\varphi(N)$. Choose a complement $C$ to $\mathrm{d}\varphi|_{N}(E)\oplus T\varphi(N)$ in $TM|_{\varphi(N)}$, i.e.
		\[
		TM|_{\varphi(N)}=T\varphi(N)\oplus \mathrm{d}\varphi|_{N}(E)\oplus C.
		\]
		Fix a linear connection $\nabla$ on $TM$, and define a map
		\[
		\psi:E\oplus (\varphi|_{N})^{*}C\rightarrow M:(e,c)\rightarrow \exp_{\nabla}\big(Tr_{\varphi(te)}c\big),
		\]
		where $Tr_{\varphi(te)}$ denotes parallel transport along the curve $t\mapsto \varphi(te)$ for $t\in[0,1]$. We slightly abuse notation, since the map $\psi$ is only defined on a small enough neighborhood of the zero section $N$. Clearly, $\psi$ satisfies the following properties:
		\begin{itemize}
			\item $\psi$ restricts to $\varphi|_{N}$ along the zero section $N$.
			\item For $p\in N$ and a vertical tangent vector $(e,c)\in T_{p}\big(E\oplus (\varphi|_{N})^{*}C\big)$, we have
			\begin{align*}
				(\mathrm{d}\psi)_{p}(e,c)&=\left.\frac{\mathrm{d}}{\mathrm{d}s}\right|_{s=0}\psi(s e,0)+\left.\frac{\mathrm{d}}{\mathrm{d}s}\right|_{s=0}\psi(0,s c)\\
				&=\left.\frac{\mathrm{d}}{\mathrm{d}s}\right|_{s=0}\exp_{\nabla}\big(0_{\varphi(se)}\big)+\left.\frac{\mathrm{d}}{\mathrm{d}s}\right|_{s=0}\exp_{\nabla}(sc)\\
				&=\left.\frac{\mathrm{d}}{\mathrm{d}s}\right|_{s=0}\varphi(se)+c\\
				&=(\mathrm{d}\varphi)_{p}(e)+c,
			\end{align*}
			which shows that $\mathrm{d}\psi$ is an isomorphism at points of the zero section.
			\item We have that $\psi(e,0)=\varphi(e)$, i.e. the following diagram commutes:
			\begin{equation}\label{diag}
				\begin{tikzcd}[column sep=large, row sep=large]
					& E\oplus(\varphi|_{N})^{*}C\arrow[d,"\psi"] \\
					E\arrow[r,"\varphi"]\arrow[ru,hookrightarrow] &  M 
				\end{tikzcd}
			\end{equation}
		\end{itemize}
		Using the first and second bullet point above, the inverse function theorem for submanifolds (e.g. \cite[Lemma 6.1.3]{mukherjee}) shows that $\psi$ is an embedding on a neighborhood of $N$. Since also the inclusion $E\hookrightarrow E\oplus(\varphi|_{N})^{*}C$ on the left in \eqref{diag} is an embedding, it follows that $\varphi$ is an embedding on a neighborhood of $N$ in $E$.
	\end{proof}

	\begin{remark}\label{nbhd}
		If a map $\varphi:U\subset E\rightarrow M$ satisfying the assumptions \eqref{assumptions} of Proposition \ref{embedding} is only defined on a neighborhood $U\subset E$ of $N$, then the conclusion of the proposition still holds. This can be obtained, for instance, by constructing a smooth map $\mu:E\rightarrow E$ such that $\mu(E)\subset U$ and $\mu=\text{Id}$ near $N$ (see \cite[Chapter 4, \S5]{hirsch}). Then Proposition \ref{embedding} implies that the composition $\varphi\circ\mu:E\rightarrow M$ is an embedding on a neighborhood of $N$, hence the same holds for $\varphi$.
	\end{remark}

\end{document}